\documentclass[11pt,reqno]{amsart}

\usepackage[T1]{fontenc}
\usepackage{lmodern}
\usepackage{microtype}
\usepackage[margin=1.03in]{geometry}
\usepackage{amsmath,amssymb,amsfonts,mathtools}
\usepackage{amsthm}
\usepackage{booktabs}
\usepackage{enumitem}
\usepackage{xcolor}
\usepackage[colorlinks=true,
  linkcolor=blue!50!black,
  citecolor=blue!50!black,
  urlcolor=blue!50!black]{hyperref}
\usepackage[capitalize,noabbrev]{cleveref}

\allowdisplaybreaks
\numberwithin{equation}{section}
\setlist[enumerate]{leftmargin=2.2em,itemsep=0.25em,topsep=0.4em}
\setlist[itemize]{leftmargin=2em,itemsep=0.25em,topsep=0.4em}

\theoremstyle{plain}
\newtheorem{theorem}{Theorem}[section]
\newtheorem{proposition}[theorem]{Proposition}
\newtheorem{lemma}[theorem]{Lemma}
\newtheorem{corollary}[theorem]{Corollary}

\theoremstyle{definition}
\newtheorem{definition}[theorem]{Definition}

\theoremstyle{remark}
\newtheorem{remark}[theorem]{Remark}

\title[Bouvier's conjecture and dimension sequences of UFDs]{Bouvier's conjecture and dimension sequences of unique factorization domains}

\author{Viet-Hoang Tran}
\address{Department of Mathematics, National University of Singapore, Singapore 119076}
\email{hoang.tranviet@u.nus.edu}
\urladdr{https://vh-tran.github.io/}

\author{Thieu N. Vo}
\address{Department of Computer Science, University of Bath, United Kingdom}
\email{ntv22@bath.ac.uk}

\author{Tan M. Nguyen}
\address{Department of Mathematics, National University of Singapore, Singapore 119076}
\email{tanmn@nus.edu.sg}
\urladdr{https://tanmnguyen89.github.io/}

\date{}


\hypersetup{
  pdftitle={Bouvier's conjecture and dimension sequences of UFDs},
  pdfauthor={Viet-Hoang Tran}
}

\begin{document}
\raggedbottom

\begin{abstract}
We prove Bouvier's conjecture. More generally, an integer sequence
\((a_n)_{n\ge0}\) with \(a_0=d\ge0\) is realized by a unique factorization domain (UFD) \(R\)
with \(\dim R[X_1,\ldots,X_n]=a_n\) for every \(n\ge0\) if and only if
\[
 a_n+1\le a_{n+1}\le
 a_n+\left\lfloor\frac{a_n+1}{n+1}\right\rfloor
 \qquad(n\ge0)
\]
and \(a_1\le2d\) whenever \(d\ge1\).
\end{abstract}

\maketitle

\section{Introduction}\label{sec:introduction}

Write \(A[n]=A[X_1,\ldots,X_n]\), with \(A[0]=A\).
The dimension-sequence problem asks which integer sequences
\((a_n)_{n\ge0}\) occur as \(a_n=\dim A[n]\).
Seidenberg proved that a commutative ring of finite dimension \(d\)
satisfies the bounds \cite{Seidenberg1953}
\[
 d+1\le\dim A[X]\le2d+1
\]
and subsequently realized every integer pair in this range by an integrally closed domain
\cite[Theorem~3]{Seidenberg1954}.
Arnold and Gilmer solved the unrestricted sequence problem
\cite{ArnoldGilmer1973,ArnoldGilmer1974}: its solutions are finite
coordinatewise maxima of sequences \((b_n)\) with positive
nonincreasing successive differences and \(b_1-b_0\le b_0+1\).
Parker gave the numerical characterization \cite{Parker1975}:
a sequence \((a_n)_{n\ge0}\) is realized by a nonzero commutative
ring if and only if \(a_0\ge0\) and
\begin{equation}\label{eq:P}
 a_n+1\le a_{n+1}\le
 a_n+\left\lfloor\frac{a_n+1}{n+1}\right\rfloor
 \qquad(n\ge0).
\end{equation}

For a domain \(A\), the valuative dimension \(\dim_{\mathrm v}A\)
is the supremum of the dimensions of its valuation overrings.
Jaffard's theory identifies the finite-dimensional domains satisfying
\(\dim_{\mathrm v}A=\dim A\) with those for which
\(\dim A[n]=\dim A+n\) for every \(n\)
\cite{Jaffard1960}; these are called \emph{Jaffard domains}.
According to Fontana and Kabbaj, Bouvier conjectured in a 1985
course that finite-dimensional Krull domains, and even UFDs, need
not be Jaffard \cite{FontanaKabbaj2004}.
Bouvier and Kabbaj published the factorial existence problem and
conjectured that every positive integer occurs as the difference
between valuative and Krull dimension of a UFD
\cite[p.~163]{BouvierKabbaj1988}.
They also asked for a finite-dimensional UFD that is not a
strong \(S\)-domain.
Here an \(S\)-domain is a domain whose height-one primes retain
height one under polynomial extension in one indeterminate;
a strong \(S\)-domain has this property in every prime quotient.

Bouchiba and Kabbaj surveyed known families of finite-dimensional
non-Noetherian Krull domains; the examples considered there were
Jaffard and often locally Jaffard
\cite{BouchibaKabbaj2009}.
The one-variable existence problem appears as Problem~37a in
\cite{CahenFontanaFrischGlaz2014}: can a Krull domain, or a UFD,
satisfy \(\dim A[X]>\dim A+1\)?
This is equivalent to the non-Jaffard existence problem: if
\(m\) is the first index of excess growth, replace \(A\) by
\(A[m-1]\), which retains the Krull or UFD property.

We prove Bouvier's conjecture by constructing a two-dimensional
local UFD \(A\) with \(\dim A[X]=4\). To our knowledge, these
are the first examples of finite-dimensional non-Jaffard Krull
domains, and they satisfy the stronger UFD requirement.
More generally, we solve the full dimension-sequence problem for UFDs.
Besides \eqref{eq:P},
a UFD of dimension \(d\ge1\) necessarily satisfies
\begin{equation}\label{eq:K}
 a_1\le2d,
\end{equation}
(see \cref{prop:necessary}). These conditions give an exact characterization.

\begin{theorem}\label{thm:main}
Let \((a_n)_{n\ge0}\) be a sequence of integers with
\(a_0=d\ge0\). There exists a UFD \(R\) such that
\[
 \dim R[X_1,\ldots,X_n]=a_n\qquad(n\ge0)
\]
if and only if
\[
 a_n+1\le a_{n+1}\le
 a_n+\left\lfloor\frac{a_n+1}{n+1}\right\rfloor
 \quad(n\ge0),
 \qquad a_1\le2d\ \text{if }d\ge1.
\]
\end{theorem}

For \(d=0\), the first inequality forces \(a_n=n\), which is realized
by a field. For \(d\ge1\), the unrestricted sequences excluded
by the second inequality are exactly those with \(a_1=2d+1\).
Thus the change from the classical problem is confined to the
first-step bound.

\begin{corollary}[Bouvier's conjecture]\label{cor:bouvier}
There exists a local UFD \(A\) such that
\[
 \dim A=2,\qquad \dim A[X]=4.
\]
In particular, finite-dimensional non-Jaffard Krull domains exist.
\end{corollary}

The constructions also realize later excess growth after a minimal
first increase. More precisely, for every \(r\ge1\) there is
a two-dimensional local UFD \(A\) with
\[
 \dim A[X]=3,\qquad \dim_{\mathrm v}A=2+r;
\]
there is such an example with infinite valuative dimension as well
(\cref{cor:bouvier-defect}). Thus we also prove the stronger
quantitative conjecture of Bouvier and Kabbaj. The full range of two-dimensional
sequences is described in \cref{cor:two-dimensional-minimal,cor:two-dimensional-full}.

The proof first expresses polynomial dimension by residue
transcendence along prime flags and reduces the numerical conditions
to finite maxima of basic profiles. We then construct the basic
UFDs and realize their finite maxima by faithfully flat extensions
and birational gluing.
Throughout, rings are nonzero and commutative with identity,
homomorphisms preserve identity, and all dimension-sequence entries
are finite. For an infinite cap we set \(\min(n,\infty)=n\).

\section{Prime chains and residue transcendence}
\label{sec:background}

For a prime \(\mathfrak p\) of a domain \(A\), write
\(\kappa(\mathfrak p)=\operatorname{Frac}(A/\mathfrak p)\).
A valuation ring \(V\) \emph{dominates} a local domain \((B,\mathfrak m)\)
if \(B\subseteq V\) and the maximal ideal of \(V\) contracts to
\(\mathfrak m\). Every local domain is dominated by a valuation ring
of its fraction field.

We use this existence theorem, extension of
valuations to field extensions, and the correspondence between valuation
overrings and convex subgroups of the value group in their usual forms
\cite[Ch.~VI]{BourbakiCA}. No restriction on valuation rank is imposed.

We shall also use the elementary facts that faithfully flat maps are
surjective on spectra and satisfy going down, and that integral maps
satisfy lying over, going up, and incomparability
\cite[Chs.~3, 5]{AtiyahMacdonald1969}. For a prime \(Q\) in a
polynomial ring over a field \(K\),
\begin{equation}\label{eq:field-height}
 \operatorname{ht}Q+\operatorname{trdeg}_K
       \operatorname{Frac}(K[n]/Q)=n.
\end{equation}
This is the usual dimension theorem for affine domains over a field.

\begin{lemma}[Residues under field extension]\label{lem:residue-extension}
Let \(K\subseteq E\) be fields, let \(W\) be a valuation ring of
\(E\), and put \(V=W\cap K\). Then
\begin{equation}\label{eq:residue-extension}
 \operatorname{trdeg}_{\kappa(V)}\kappa(W)
       \leq\operatorname{trdeg}_K E.
\end{equation}
In particular an algebraic field extension induces an algebraic residue
extension. If \(k\subseteq E\), the valuation is trivial on \(k\),
and some element of \(E\) has positive value, then
\begin{equation}\label{eq:positive-value-bound}
 1+\operatorname{trdeg}_k\kappa(W)
       \leq\operatorname{trdeg}_k E.
\end{equation}
\end{lemma}

\begin{proof}
Lift a finite tuple of residues algebraically independent over
\(\kappa(V)\) to units of \(W\). In a polynomial relation over
\(K\), divide the coefficients by one of least value. The resulting
coefficients lie in \(V\), and at least one is a unit. Reduction gives
a nonzero relation over \(\kappa(V)\), a contradiction. This proves
\eqref{eq:residue-extension}, including infinite transcendence degrees.
For \eqref{eq:positive-value-bound}, choose an element \(z\) of
positive value and unit lifts of a residue tuple independent over \(k\).
In a relation grouped by powers of \(z\), every nonzero coefficient
polynomial has value zero. The first nonzero term then has strictly
least value, so cancellation is impossible.
\end{proof}

\begin{definition}\label{def:capacity}
For a domain \(A\) and primes \(\mathfrak p\subsetneq\mathfrak q\),
define the \emph{residue capacity}
\[
 \tau_A(\mathfrak p,\mathfrak q)
 =\sup_V\operatorname{trdeg}_{\kappa(\mathfrak q)}\kappa(V)
 \in\mathbb N_0\cup\{\infty\},
\]
where \(V\) ranges over the valuation rings of
\(\kappa(\mathfrak p)\) dominating
\((A/\mathfrak p)_{\mathfrak q/\mathfrak p}\). Domination supplies
the indicated inclusion of residue fields. We set
\(\min(n,\infty)=n\).
\end{definition}

\begin{theorem}[The prime-flag formula]\label{thm:capacity}
For every finite-dimensional domain \(A\) and every \(n\geq0\),
\begin{equation}\label{eq:capacity-formula}
 \dim A[n]=n+\max_{0=\mathfrak p_0\subsetneq\cdots\subsetneq\mathfrak p_r}
 \left\{r+\sum_{i=0}^{r-1}
        \min\bigl(n,\tau_A(\mathfrak p_i,\mathfrak p_{i+1})\bigr)\right\}.
\end{equation}
The flag consisting only of zero is allowed. Flags need not be saturated.
\end{theorem}

\begin{proof}
Fix a flag and put
\(h_i=\min(n,\tau_A(\mathfrak p_i,\mathfrak p_{i+1}))\).
Choose a dominating valuation whose residue field contains \(h_i\)
elements algebraically independent over \(\kappa(\mathfrak p_{i+1})\),
and choose unit lifts \(u_1,\ldots,u_{h_i}\). Only a finite tuple is
needed, even for infinite capacity. In \((A/\mathfrak p_i)[n]\),
evaluate the first \(j\) variables at \(u_1,\ldots,u_j\), leaving
the others indeterminate. The kernels, for \(0\leq j\leq h_i\),
are a strict prime chain contracting to zero. Indeed, writing
\(u_j=b_j/a_j\) in \(\kappa(\mathfrak p_i)\), the polynomial
\(a_jX_j-b_j\) belongs to the \(j\)-th kernel and not the previous one.
Reduction in the valuation ring shows that every kernel lies in
\((\mathfrak p_{i+1}/\mathfrak p_i)[n]\): independence of the chosen
residues forces every coefficient of an evaluation identity to vanish
modulo \(\mathfrak p_{i+1}/\mathfrak p_i\). The final containment is
strict, since its nonzero constants are absent from the kernels.
Lifting and concatenating gives \(h_i+1\) inclusions along each flag
edge. Append the \(n\) variable primes above \(\mathfrak p_r[n]\).
This proves the lower bound.

For the upper bound, prepend zero to any finite prime chain in \(A[n]\)
and group its primes by their distinct contractions
\(0=\mathfrak p_0\subsetneq\cdots\subsetneq\mathfrak p_r\).
In the \(i\)-th block let \(\alpha_i,\beta_i\) be the heights of
its first and last primes after quotienting by \(\mathfrak p_i[n]\)
and passing to \(\kappa(\mathfrak p_i)[n]\). These heights lie
between zero and \(n\), and \(\alpha_0=0\). If the chain has length
\(L\), then
\begin{equation}\label{eq:chain-telescope}
 L\leq r+\sum_{i=0}^r(\beta_i-\alpha_i)
 \leq r+n+\sum_{i=0}^{r-1}\max(0,\beta_i-\alpha_{i+1}).
\end{equation}
Let \(P\) be the last prime of block \(i\) and \(Q\) the first
prime of block \(i+1\). A valuation ring \(W\) of
\(E=\operatorname{Frac}(A[n]/P)\) dominates
\((A[n]/P)_{Q/P}\). Its restriction \(V\) to
\(K=\kappa(\mathfrak p_i)\) dominates
\((A/\mathfrak p_i)_{\mathfrak p_{i+1}/\mathfrak p_i}\).
The dominated residue field \(\operatorname{Frac}(A[n]/Q)\)
embeds in \(\kappa(W)\). By \eqref{eq:field-height},
\[
 \operatorname{trdeg}_K E=n-\beta_i,
 \qquad
 \operatorname{trdeg}_{\kappa(\mathfrak p_{i+1})}
       \operatorname{Frac}(A[n]/Q)=n-\alpha_{i+1}.
\]
The residue-extension inequality and the transcendence-degree tower
formula give
\[
 \operatorname{trdeg}_{\kappa(\mathfrak p_{i+1})}\kappa(V)
       \geq\beta_i-\alpha_{i+1}.
\]
If the left side is infinite the assertion is immediate; otherwise
all terms needed in this comparison are finite. Hence
\[
 \max(0,\beta_i-\alpha_{i+1})
 \leq\min\bigl(n,\tau_A(\mathfrak p_i,\mathfrak p_{i+1})\bigr).
\]
Substitute in \eqref{eq:chain-telescope}. Finally \(r\leq\dim A\)
and each summand is at most \(n\), so the integer flag values are
bounded and their supremum is a maximum.
\end{proof}

\begin{proposition}[Necessary numerical inequalities]\label{prop:necessary}
Let \(A\) be a nonzero ring of finite dimension \(d\), and put
\(a_n=\dim A[n]\). Then
\begin{equation}\label{eq:parker-necessary}
 a_n+1\leq a_{n+1}\leq
 a_n+\left\lfloor\frac{a_n+1}{n+1}\right\rfloor
 \qquad(n\geq0).
\end{equation}
If \(A\) is Krull and \(d\geq1\), then additionally
\begin{equation}\label{eq:krull-first-bound}
 a_1\leq2d.
\end{equation}
For \(d=0\) one has \(a_n=n\).
\end{proposition}

\begin{proof}
For a domain the prime-flag formula expresses \(e_n=a_n-n\) as a
maximum of functions \(h_n=r+\sum_i\min(n,t_i)\). Each is
nondecreasing, and \(h_n/(n+1)\) is nonincreasing, since each
\((1+\min(n,t_i))/(n+1)\) is. Thus
\[
 h_n\leq h_{n+1}\leq h_n+\left\lfloor\frac{h_n}{n+1}\right\rfloor.
\]
Taking maxima preserves this inequality because its right-hand function
of \(h_n\) is nondecreasing. Substitution of \(e_n=a_n-n\) proves
\eqref{eq:parker-necessary} for domains.

For a general ring,
\begin{equation}\label{eq:minimal-prime-dimensions}
 \dim A[n]=\sup_{\mathfrak p\in\operatorname{Min}A}
                    \dim(A/\mathfrak p)[n].
\end{equation}
Indeed, the contraction of the first member of a finite polynomial
prime chain contains a minimal prime \(\mathfrak p\); the entire
chain then survives modulo \(\mathfrak p[n]\). The reverse inequality
is immediate. These dimensions are finite: grouping a chain by
contractions gives at most \(d+1\) blocks, each with at most \(n+1\)
members, and hence
\begin{equation}\label{eq:seidenberg-many}
 d+n\leq\dim A[n]\leq(d+1)(n+1)-1.
\end{equation}
The lower bound extends a length-\(d\) base chain and appends the
variable primes. Applying the already proved domain inequalities to
\eqref{eq:minimal-prime-dimensions} proves
\eqref{eq:parker-necessary}; the same bound gives the assertion for
\(d=0\).

Now suppose \(A\) is a positive-dimensional Krull domain. A
one-variable chain has at most two primes in each contraction block.
If its length were \(2d+1\), the bound just proved would force exactly
\(d+1\) blocks, all doubled. Their contractions form a length-\(d\)
chain starting at zero. The second contraction \(\mathfrak p\) has
height one, since an intermediate nonzero prime would make a base
chain of length \(d+1\). The first four polynomial primes remain
distinct after localization at \(\mathfrak p\). They would give
a length-three chain in \(A_{\mathfrak p}[X]\), whereas
\(A_{\mathfrak p}\) is a DVR and its polynomial ring has dimension
two. This contradiction proves \eqref{eq:krull-first-bound}.
\end{proof}

\begin{remark}\label{rem:krull-preliminaries}
A Krull domain is the intersection, in its fraction field, of its
height-one DVR localizations, and every nonzero element is a nonunit
in only finitely many of them. We use this finite-character
representation, its preservation under localization and polynomial
extension, and the criterion that a Krull domain is a UFD precisely
when every height-one prime is principal \cite{Gilmer1972}.
For a height-one prime \(\mathfrak p\),
\(\tau_A(0,\mathfrak p)=0\): a valuation ring of the same field
dominating the DVR \(A_{\mathfrak p}\) equals that DVR. Indeed its
units have value zero and a uniformizer has positive value, which
determines membership for every fraction.
Applying \eqref{eq:krull-first-bound} to \(A[n]\) gives
\(a_{n+1}\leq2a_n\) whenever \(a_n>0\). For \(n\geq1\), however,
this already follows from \eqref{eq:parker-necessary}; only the first
bound strengthens the unrestricted numerical conditions.
\end{remark}

\begin{remark}[Valuative dimension]\label{rem:jaffard}
The valuative dimension \(\dim_{\mathrm v}A\) is the supremum of
the dimensions of valuation overrings of \(A\). Jaffard's dimension
theorems give
\[
 \dim_{\mathrm v}A=\sup_{n\geq0}(\dim A[n]-n),\qquad
 \dim_{\mathrm v}A[n]=\dim_{\mathrm v}A+n
\]
\cite[Ch.~IV, Theorems~2, 4, and~5]{Jaffard1960}.
If \(\dim_{\mathrm v}A=v<\infty\), then
\(\dim A[n]=n+v\) for all sufficiently large \(n\). A
finite-dimensional domain is called \emph{Jaffard} if
\(\dim_{\mathrm v}A=\dim A\), equivalently if
\(\dim A[n]=\dim A+n\) for every \(n\). These classical valuative
statements are used only to interpret the consequences for Bouvier's
conjecture; the sequence classification uses the proved prime-flag
formula directly.
\end{remark}

\section{Numerical reduction}
\label{sec:numerical}

For \(r\ge0\) and \(t_1,\ldots,t_r\in\mathbb N_0\cup\{\infty\}\), put
\begin{equation}\label{eq:basic-unrestricted}
 F_{r,\mathbf t}(n)=n+r+\sum_{i=1}^r\min(n,t_i).
\end{equation}
The empty list gives \(F_{0,\varnothing}(n)=n\).

\begin{proposition}\label{prop:envelope}
A sequence of integers with \(a_0\ge0\) satisfies \eqref{eq:P}
if and only if it is the pointwise maximum of finitely many
profiles \eqref{eq:basic-unrestricted}.
\end{proposition}
\begin{proof}
Set \(e_n=a_n-n\). The recurrence is equivalent to
\begin{equation}\label{eq:excess-recurrence}
 e_n\le e_{n+1}\le e_n+\left\lfloor\frac{e_n}{n+1}\right\rfloor.
\end{equation}
For a basic profile its excess is \(r+\sum_i\min(n,t_i)\).
It is nondecreasing, and its quotient by \(n+1\) is nonincreasing,
since each \((1+\min(n,t_i))/(n+1)\) is nonincreasing.
Thus it satisfies \eqref{eq:excess-recurrence}. Finite maxima
preserve that recurrence because
\(x\mapsto x+\lfloor x/(n+1)\rfloor\) is increasing.

Conversely, for each \(N\ge0\), write
\(e_N=q(N+1)+b\), with \(0\le b\le N\), and set
\begin{equation}\label{eq:finite-support}
 h^N_m=q(\min(m,N)+1)+\min(\min(m,N)+1,b).
\end{equation}
This is a basic excess: take \(q\) caps equal to \(N\) and,
if \(b>0\), one additional cap \(b-1\). It agrees with \(e_N\)
at \(N\). For \(m\ge N\), monotonicity gives \(h^N_m\le e_m\).
For \(m<N\), use backwards induction and
\[
 e_m\ge\left\lceil\frac{m+1}{m+2}e_{m+1}\right\rceil,
 \qquad
 \left\lceil\frac{m+1}{m+2}
 [q(m+2)+\min(m+2,b)]\right\rceil
 =q(m+1)+\min(m+1,b).
\]
Hence \(h^N_m\le e_m\) for every \(m\).

The nonnegative integers \(k_n=\lfloor e_n/(n+1)\rfloor\)
are nonincreasing, so eventually equal \(k\). On that tail,
\(e_{n+1}-e_n\le k\), so the nonnegative integers
\(b_n=e_n-k(n+1)\) satisfy
\(b_{n+1}-b_n=e_{n+1}-e_n-k\le0\) and eventually equal a
constant \(b\). Therefore
\[
 e_n=k(n+1)+b\quad(n\gg0),\qquad
 h^\infty_m=k(m+1)+\min(m+1,b)
\]
defines a basic excess with \(k\) infinite caps and, if \(b>0\),
one cap \(b-1\). For fixed \(m\), it equals \(h^N_m\) for all
sufficiently large \(N\), so it lies below \(e_m\). It agrees
with \(e_m\) eventually. Adjoining the finitely many supports
\eqref{eq:finite-support} for the remaining indices gives \(e\)
as their maximum. Adding \(m\) proves the assertion.
\end{proof}

The profiles needed for factorial realization are
\begin{align}
 O_{r,\mathbf t}(n)&=n+r+1+\sum_{i=1}^r\min(n,t_i),
 &&r\ge0,\label{eq:ordinary-profile}\\
 B_{r,\mathbf t}(n)&=n+\max\left\{r+1,
 r+\sum_{i=1}^r\min(n,t_i)\right\},
 &&r\ge1.\label{eq:delayed-profile}
\end{align}

\begin{corollary}\label{cor:krull-envelope}
If \(a_0=d\ge1\) and \eqref{eq:P}--\eqref{eq:K} hold,
then \(a\) is a finite maximum of ordinary profiles
\eqref{eq:ordinary-profile} and delayed profiles
\eqref{eq:delayed-profile}. Delayed profiles with a zero cap can
be replaced by ordinary profiles and a dimension baseline.
\end{corollary}
\begin{proof}
In a supporting basic profile \(F_{r,\mathbf t}\le a\), one
has \(r\le d\). If \(r=d\), then
\(d+1+\#\{i:t_i>0\}=F_{d,\mathbf t}(1)\le2d\); deleting
one zero cap gives an ordinary profile. If \(r<d\), then
\[
 \max\{n+d,F_{r,\mathbf t}(n)\}
 =\max\{n+d,B_{r,\mathbf t}(n)\}.
\]
The baseline is below \(a_n\) by \eqref{eq:P}, and a support
with \(r=0\) is redundant. Finally, if a delayed profile has
at most \(r-1\) positive caps, pad them with zeros to length
\(r-1\). Its formula is the maximum of the resulting ordinary
profile \(n+r+\sum_i\min(n,t_i)\) and \(n+r+1\).
\end{proof}

\begin{corollary}\label{cor:growth}
Every unrestricted profile with initial dimension \(d\) is
eventually affine with integer slope in \(\{1,\ldots,d+1\}\).
For \(d\ge1\), under \eqref{eq:K} the slope lies in \(\{1,\ldots,d\}\) and
\(d+n\le a_n\le d(n+1)\).
\end{corollary}
\begin{proof}
The stabilization in \cref{prop:envelope} gives the affine tail.
The bound \(k\le e_0=d\) gives its first slope range.
For the Krull upper bound, start at \(a_1\le2d\) and, for
\(n\ge1\), use
\[
 a_{n+1}\le d(n+1)+
 \left\lfloor\frac{d(n+1)+1}{n+1}\right\rfloor=d(n+2).
\]
This excludes slope \(d+1\). The lower bound follows by induction
from the left side of \eqref{eq:P}.
\end{proof}

\section{Ordinary profiles with arithmetic control}\label{sec:ordinary}

We call a field \(F\supset\mathbb Q\) \emph{rationally Hilbertian} if
Hilbert irreducibility over \(F\), with finitely many parameters and
finitely many nonzero polynomial exclusions, admits parameter values
in \(\mathbb Q\). The following construction supplies the coefficient
rings required later.

\begin{theorem}\label{thm:ordinary}
Let \(m\ge1\), let \(d_1,\ldots,d_m\) be nonnegative integers, and let
\(t_{j,i}\in\mathbb N_0\cup\{\infty\}\), \(1\le i\le d_j\).
There is a countable semilocal UFD \(C\supset\mathbb Q\), with maximal
ideals \(\mathfrak m_1,\ldots,\mathfrak m_m\), and prime elements
\(\pi_j\in\mathfrak m_j\setminus\bigcup_{k\ne j}\mathfrak m_k\), such that,
writing \(C_j=C_{\mathfrak m_j}\), one has
\begin{equation}\label{eq:ordinary-realization}
 \dim C_j[n]=n+d_j+1+\sum_{i=1}^{d_j}\min(n,t_{j,i})
            =1+\dim(C_j/\pi_jC_j)[n]\qquad(n\ge0).
\end{equation}
Each \(C_j[1/\pi_j]\) is a PID, and \(\operatorname{Frac}C\) has a
discrete valuation \(\nu_j\) dominating \(C_j\), with
\(\nu_j(\pi_j)=1\). The residue field of \(C_j\) is \(\mathbb Q\).
The common fraction field and every nonzero prime quotient field of
every \(C_j[1/\pi_j]\) are rationally Hilbertian.
\end{theorem}
The remainder of this section is devoted to the proof of this theorem.

\subsection{Two arithmetic lemmas}

\begin{lemma}\label{lem:ordinary-rational-hilbert}
Every finitely generated field extension of \(\mathbb Q\) is
rationally Hilbertian.
\end{lemma}
\begin{proof}
Let \(F/\mathbb Q(u_1,\ldots,u_s)\) be finite separable of degree \(e\),
and let \(f(T,X)\in F(T)[X]\) be irreducible of degree \(h\).
Choose \(\alpha\) with \(F=\mathbb Q(u)(\alpha)\), and let
\(\beta\) be a root of \(f\). The extension \(F(T)(\beta)\)
has degree \(eh\) over \(\mathbb Q(u,T)\). Choose a primitive
element \(\theta=\alpha+\lambda\beta\), with \(\lambda\in\mathbb Q\),
and monic polynomial \(H(u,T,Z)\). Invert a single nonzero
polynomial \(b(u,T)\) so that
the coefficients of \(H\), the identities expressing the generators
of \(F\) and \(\beta\) in \(\theta\), and the reverse identity
\(\theta=\alpha+\lambda\beta\) are defined. After clearing
denominators, include in \(b\) the norms of all denominators,
of the leading coefficient of \(f\), and of each prescribed
nonzero exclusion. Classical multivariable Hilbert irreducibility
over \(\mathbb Q\) gives rational \((u_0,c)\) with \(b(u_0,c)\ne0\)
and \(H(u_0,c,Z)\) irreducible of degree \(eh\)
\cite{Iadarola2021}.
Then \(H(u,c,Z)\) is irreducible over \(\mathbb Q(u)\): a monic
factorization has coefficients in the integrally closed ring
\(\mathbb Q[u,1/b(u,c)]\) and would specialize at \(u_0\).
The retained identities identify this degree-\(eh\) field with
\(F(\beta_c)\), where \(f(c,\beta_c)=0\). The tower degree
formula proves irreducibility
of \(f(c,X)\) over \(F\). The same proof handles several parameters
and finite simultaneous conditions.
\end{proof}

\begin{lemma}\label{lem:ordinary-scalar-lock}
Let \(F_\alpha\) be finitely many finitely generated fields over
\(\mathbb Q\), each containing \(\mathbb Q(p)\), and let
\(L_\alpha/F_\alpha\) be finite extensions. Prescribe finitely many
nonzero elements \(\rho\) of each \(F_\alpha\).
For all but finitely many \(a\in\mathbb Q\), each \(L_\alpha\) has
a discrete valuation with
\[
 v(p-a)=1,\qquad v(\rho)=0.
\]
Consequently, for any \(b\ge2\), adjoining a \(b\)-th root of
\(\rho/(p-a)\) is linearly disjoint from \(L_\alpha/F_\alpha\).
\end{lemma}
\begin{proof}
Extend \(p\) to a transcendence basis of \(F_\alpha/\mathbb Q\).
Over the resulting rational field, \(L_\alpha\) is finite separable.
A primitive element, its discriminant, and denominator clearing give
a finite etale model after inverting finitely many polynomials.
Arrange also that each prescribed \(\rho\) and \(\rho^{-1}\) is
integral over the localized base. For all but finitely many rational
\(a\), none of the inverted polynomials vanishes identically at
\(p=a\). The \((p-a)\)-adic Gauss DVR then has an unramified
prolongation to \(L_\alpha\), and every prescribed \(\rho\) is a unit.
This is the usual discriminant proof of unramifiedness away from
finitely many divisors; see \cite{BourbakiCA}.
At that DVR, \(X^b-(p-a)/\rho\) is Eisenstein. Its irreducibility
over both fields gives degree \(b\) after either base extension,
which is the asserted linear disjointness.
\end{proof}

\subsection{Polynomial carriers for the boundary}

Here is the characteristic-zero carrier used at one level. Let \(D\)
be a countable local domain, \(L=\operatorname{Frac}D\), and
\(K=L(s_1,\ldots,s_t)\), allowing countably many independent \(s_i\).
Use rational DVR fields \(K(U_r)\), with arrows
\begin{equation}\label{eq:ordinary-boundary-arrow}
 U_r=q(s)U_{r+1}^{b},\qquad s_i\longmapsto s_i,
\end{equation}
where \(b\) is an arbitrarily large power of two. At a finite stage,
only finitely many \(s_i\) occur. The carrier coordinates are
\(U_r,s_1U_r,\ldots,s_tU_r\). If \(q=H^2\) and
\(b\ge2\deg H+3\), their images are polynomials of degree at least
two in the new coordinates.

Every finite collection of positive elements of the old DVR fits
in a later carrier. Indeed such an element has the form
\[
 \frac{\sum_{\ell\ge1}a_\ell(s)U_r^\ell}
      {1+\sum_{\ell\ge1}b_\ell(s)U_r^\ell}.
\]
Choose \(H\) divisible by all coefficient denominators. The square
\(H^2\) still clears them, and sufficiently large \(b\) homogenizes
every positive term. The denominator becomes one plus an element
of the new coordinate ideal, hence a local unit. Coefficients in
\(L\) are cleared as follows: choose \(0\ne c\in D\) clearing
the finitely many coefficients, and replace the new uniformizer
by its quotient by \(c\). A coefficient in degree \(N>0\) is
multiplied by \(c^N\). Since \(b\) is even, the factor \(c^b\)
in \eqref{eq:ordinary-boundary-arrow} remains a square.
Only elements of \(D\), not inverses of its nonunits, are put in
a later lower carrier.

Fairly performing these finite tasks produces a rank-one valuation
ring \(V\), with nondiscrete dyadic value group and residue field \(K\),
and exactly the boundary ring \(D+M\), where \(M\) is the maximal
ideal of \(V\). To see equality, the carrier union lies in \(D+M\),
and every element of \(M\) occurs in a finite DVR chart and is
eventually included by the preceding calculation. The union DVR
fields form a valuation field because the sign of every element
is decided in one chart. Their extensions are totally ramified,
so the residue field is \(K\).

These constructions may be nested. Write \(D_0=L_0=\mathbb Q\) and
\begin{equation}\label{eq:ordinary-boundary-tower}
 D_i=D_{i-1}+M_i,\qquad
 L_i=\operatorname{Frac}D_i,\qquad
 \kappa(V_i)=L_{i-1}(s_{i,1},\ldots,s_{i,t_i}).
\end{equation}
The valued fields are built together with their arrows, rather than
requiring roots of coefficients in a previously fixed field.
Each finite stage uses finitely many lower coordinates. Enlarging a
lower chart extends the upper coefficient field and leaves its
uniformizer and ratios independent. Thus every finite set of targets
fits into a finite nested polynomial carrier.

The spectrum of \(D_i\) is the chain consisting of \(0\), \(M_i\),
and the inverse images of the nonzero primes of \(D_{i-1}\).
Indeed every nonzero element of \(M_i\), when inverted, inverts
all of \(M_i\) by the Archimedean value group and leaves the upper
fraction field; primes containing \(M_i\) are the pullbacks of
the lower primes. On a finite interval of levels,
\begin{equation}\label{eq:ordinary-boundary-td}
 \operatorname{trdeg}_{L_a}L_b
       =\sum_{i=a+1}^{b}(1+t_i)
\end{equation}
if all caps in that interval are finite, even if \(L_a\) itself has
infinite transcendence degree. The valuation \(V_i\) supplies
capacity \(t_i\) on its individual boundary edge. Consequently
\begin{equation}\label{eq:ordinary-boundary-lower}
 \dim D_d[n]\ge n+d+\sum_{i=1}^{d}\min(n,t_i)
\end{equation}
by the evaluation-kernel lower bound in \cref{thm:capacity}.

\subsection{Simultaneous branches and exact identities}

Choose distinct rational \(c_1,\ldots,c_m\), put
\[
 \pi_j=p-c_j,\qquad P=\prod_j\pi_j,\qquad
 C_0=\mathbb Q[p]_{\mathbb Q[p]\setminus\bigcup_j(\pi_j)},
 \qquad D=\max_jd_j.
\]
We construct prefixes \(C_0\subset C_1\subset\cdots\subset C_D\).
At row \(i\), branch \(j\) is active if \(i\le d_j\), with a block
of \(1+t_{j,i}\) coordinates; an infinite cap has increasing finite
width. At branch \(j\), every other block is called silent.
Finite prefix stages are polynomial rings over \(\mathbb Q\) in
\(p\) and all current blocks up to that row, localized away from
the union of the ideals
\((p-c_j,\text{all coordinates})\).
They are semilocal UFDs. All their fraction fields are rational.

Put
\[
 b_j(p)=\prod_{k\ne j}\frac{p-c_k}{c_j-c_k}.
\]
Refine one block at a time, with higher and other current blocks
unchanged. After lower-coefficient clearing, its arrow is
\begin{equation}\label{eq:ordinary-full-arrow}
 \begin{split}
 Y_0-Pe_0&=b_j(p)(p-a)H(s)^2Z_0^b,\\
 Y_\ell-Pe_\ell&=b_j(p)(p-a)s_\ell H(s)^2Z_0^b,
       \qquad s_\ell=Z_\ell/Z_0,
 \end{split}
\end{equation}
where the \(e_\ell\) are rational digits, \(a\in\mathbb Q\setminus
\{c_1,\ldots,c_m\}\), and \(b\ge2\deg H+3\) is a power of two.
Coefficients of \(H\) belong to a later lower prefix, and no
same-row foreign block occurs in them. Modulo \(\pi_j\), this
is the carrier arrow with coefficient \((c_j-a)\bar H^2\);
at a foreign node the block becomes zero. The maps are injective
by block dominance over the lower fraction field, and are local
at each prescribed maximal ideal.

Besides boundary targets, impose the following three countable
families of requirements, each repeated cofinally.
First, for every lower element \(a_0\) divisible by none of the
\(\pi_k\), put a sufficiently large power of \(a_0\) in every
block's \(H\). Then
\begin{equation}\label{eq:ordinary-lower-action}
       Y_\ell-Pe_\ell=a_0^N\Psi_\ell(Z),
       \qquad \Psi_\ell(Z)\in(Z)^2
\end{equation}
for any prescribed finite \(N\). Even powers suffice.
Second, for a polynomial
\(f=a+\sum_{|\alpha|>0}u_\alpha Y^\alpha\) whose constant
projection at branch \(j\) is nonzero modulo \(\pi_j\), remove
all distinguished factors from \(a\), writing
\(a=a_0\prod_k\pi_k^{q_k}\). Use
\eqref{eq:ordinary-lower-action} with zero digits and large \(N\).
At branch \(j\), the other \(\pi_k\) are units, and direct division
of the positive terms by \(a\) gives
\begin{equation}\label{eq:ordinary-denominator}
                  f=a\times\text{a local unit}.
\end{equation}
Lower local denominators may first be cleared.
The finite normalization procedure below removes the distinguished
factors inside a later stage, with a permanent certificate that the
remaining factor has nonzero projection at each node.

Third, let \(f\ne0\) be a polynomial in the silent blocks at
branch \(j\), with lower coefficients. A free digit arrow makes
its constant term \(a\) nonzero: the substitution \(Y=Pe\)
is a nonzero polynomial in rational digits over the lower fraction
field. Write \(a=a_0\prod_k\pi_k^{q_k}\). On every silent block,
put a large power of \(a_0\) and a power \(\pi_j^M\) in \(H\),
and use zero digits. This is legal at that block's active node,
where \(\pi_j\) is a unit. Taking \(2M>q_j\) gives the exact
identity
\begin{equation}\label{eq:ordinary-silent}
 f=a(1+\pi_jh),\qquad
 h\in(C_i)_{\mathfrak m_{i,j}},\quad 0\ne a\in C_{i-1}.
\end{equation}
Every divided positive term contains
\(a_0^{N|\alpha|-1}\pi_j^{2M|\alpha|-q_j}\) times branch units.
Thus the parenthesized factor is a local unit. This identity
persists under all later maps. These three requirements use only
finitely many coefficients and coordinates at a time.

\subsection{Fusion, normalization, and factoriality}

We give the nonvanishing details, since mere independence of the
initial tuple does not suffice when widths increase.
For every polynomial from every created full finite stage, at
each branch \(j\), substitute the current tuple by \(PZ\) and
normalize its coefficients at the \((p-c_j)\)-adic valuation
of \(\mathbb Q(p)\). Its reduction is a nonzero polynomial
over \(\mathbb Q\). At a free sweep choose rational digits
outside these finitely many zero sets. Additional conditions,
such as the nonzero constant term used above, are likewise
nonzero polynomial conditions. A finite union of proper
polynomial zero sets cannot cover \(\mathbb Q^N\).

At branch \(j\), \(P\) has value one, and every nonconstant
term in \eqref{eq:ordinary-full-arrow} has degree at least two
in centered coordinates. Thus the current tuple is
\(Pe+O(\pi_j^2)\) for every centered future tail. A nonzero
reduction just prescribed, and its finite integer order, are
therefore permanent. A sweep through all blocks increases the
minimum difference order between two centered tails by at least
one; the triangular dependence of upper coefficients on lower
coordinates does not alter this estimate, since such terms
still contain at least two upper coordinate factors.
Insert full sweeps cofinally. Backward zero-tail truncations
then converge at every fixed stage, uniformly in future widths,
to tuples in \(\pi_j\mathbb Q[[\pi_j]]\).
Every finite-stage polynomial was treated, so evaluation is
injective and defines a valuation \(\nu_j\) on the entire union.
It has value group \(\mathbb Z\), satisfies
\(\nu_j(\pi_j)=1\), and has center exactly the prescribed
branch maximal ideal and residue field \(\mathbb Q\).

For relative fusion we must also normalize lower coefficients
at the \emph{final} DVR over \(\pi_j\). This can be done in
finitely many steps. The actual branch-\(j\) projection of a
finite prefix polynomial sets \(p=c_j\), kills all foreign
blocks, and retains the independent active carrier tuples.
Its kernel is consequently generated by \(\pi_j\) and the
foreign coordinates. A sweep makes every foreign coordinate
divisible by \(\pi_j\). Hence a zero-projection element divided
by \(\pi_j\) belongs to a later stage.
First prescribe its scalar nonvanishing condition above, fixing
a finite \(\nu_j\)-value. Repeat this division while its actual
projection is zero. The number of repetitions cannot exceed
that value. The resulting nonzero projection is preserved by
every subsequent active carrier injection. This resolves the
exact \(\pi_j\)-order of the element. Apply the same procedure
to numerators and denominators of finitely many coefficients.
Normalization at another node cannot spoil the result, since
that node's parameter is a unit here.

Now enumerate every polynomial in a current row tuple with
coefficients in every finite lower fraction-field chart.
Write \(F_{i-1}=\operatorname{Frac}C_{i-1}\) for the final constructed
lower field and \(v_0\) for its \(\pi_j\)-DVR valuation.
For \(f(Y)=\sum_\alpha c_\alpha Y^\alpha\ne0\), set
\[
 m=\min_{c_\alpha\ne0}\bigl(v_0(c_\alpha)+|\alpha|\bigr),
 \qquad g(Z)=\pi_j^{-m}f(PZ).
\]
Since \(v_0(P)=1\), the coefficients of \(g\) lie in the lower
DVR and its reduction is nonzero. The preceding finite procedure
fixes these coefficient orders and residues. Choose rational digits
\(e\) at which that reduction is nonzero. The estimate
\(Y=Pe+O(\pi_j^2)\) makes the condition permanent.
Thus the entire row tuple has an algebraically independent
fused image over \(F_{i-1}\) in the fraction field of the completion
of its \(\pi_j\)-DVR. This treats all coefficients in
\(F_{i-1}\), because each finite list occurs in a finite
lower chart. In particular it treats polynomials involving
both active and silent blocks.

All queues are countable and each requirement takes a finite
block, so a diagonal ordering treats every requirement.
For each prefix, \(\sum_j\nu_j\) is positive on every nonunit.
It bounds the number of nonunit factors of a fixed nonzero
element. A factorization of maximal length has factors
irreducible in the union. Two such factorizations lie in one
finite semilocal UFD stage, where their factors are still
irreducible, proving uniqueness. Hence every \(C_i\) is a UFD.
The branch maximal ideals have residue field \(\mathbb Q\);
an element outside their union is already a unit in a finite
stage, so these are all the maximal ideals.
Boundary cofinality proves \eqref{eq:ordinary-boundary-tower}
at every branch and
\begin{equation}\label{eq:ordinary-special-quotient}
 C_i/\pi_jC_i=D_{j,\min(i,d_j)},\qquad
 (C_i)_{(\pi_j)}\text{ is a DVR with residue field }
 L_{j,\min(i,d_j)}.
\end{equation}

For later use, upper charts over the final lower ring \(C_{i-1}\)
are polynomial localizations on independent tuples: every
alleged finite relation occurs in a later finite lower chart.
Their filtered union is flat over the lower ring, and the
prescribed maximal ideals make it faithfully flat.
If \(Q\) avoids \(P\) and has nonzero lower contraction \(q\),
use \eqref{eq:ordinary-lower-action} for \(a_0\in q\).
Every old upper expression becomes a lower expression modulo
\(qC_i\), and every finite-stage unit denominator becomes
a lower semilocal unit. Faithful flatness gives contraction
back to \(q\), and therefore
\begin{equation}\label{eq:ordinary-inherited}
       Q=qC_i,\qquad C_i/Q\simeq C_{i-1}/q.
\end{equation}

\subsection{Preserving Hilbert witnesses}

We may impose arithmetic conditions simultaneously with all
preceding choices. Write \(C_{i,s}\) for the finite chart of prefix
\(i\) at construction time \(s\). The arithmetic lemmas are
applied to fraction fields of these finite charts and their prime
quotients, all finitely generated over \(\mathbb Q\).
At every finite prefix chart, a generic prime
quotient field contains \(\mathbb Q(p)\): a nonzero polynomial
in \(p\) is a product of distinguished prime powers and a
semilocal unit, whereas a generic prime avoids \(P\).

We first justify the maps on such quotient fields. Follow a
final height-one generic prime which is new over the lower prefix.
Its irreducible representative remains irreducible at every
later finite stage: finite-stage nonunits remain nonunits.
For a block refinement, every nonunit coordinate exceptional
factor is squared, since the common factor in the
homogenized right sides of \eqref{eq:ordinary-full-arrow} is
\[
 Z_0^{\,b-2\deg H-1}(H^{\mathrm h})^2.
\]
If the old prime vanishes identically at the center \(Y=Pe\),
its image is divisible by \(Z_0^2\), impossible. A lower
coefficient exceptional factor would make the prime inherited,
also impossible for this chosen final new prime. Distinguished
factors in \(b_j\) are generic units.

Consequently this prime lies on the open where
\(D_0=Y_0-Pe_0\) and the relevant homogenized denominators
are nonzero. On that open the arrow has the finite free
presentation
\begin{equation}\label{eq:ordinary-kummer}
 Z_0^b=\frac{\rho}{p-a},\qquad
 \rho=\frac{D_0}{b_j(p)H(r)^2},\quad
 r_\ell=\frac{Y_\ell-Pe_\ell}{D_0},\quad
 Z_\ell=r_\ell Z_0.
\end{equation}
This is first a presentation over the old polynomial chart,
not an assumed quotient-field embedding. Flat going-down at
the surviving new height-one prime makes its contraction have
height at most one; it contains the old height-one prime,
so equals it \cite{AtiyahMacdonald1969}. Only now may we
pass to its old quotient field.

At each finite chart \(C_{i,s}\), fairly enumerate irreducible
polynomials over \(\operatorname{Frac}C_{i,s}\) and over its
current new height-one generic prime
quotient fields, with parameter exclusions. Skip a polynomial
if it has become reducible, and discard a prime's requirements
if it factors or becomes inherited. By
\cref{lem:ordinary-rational-hilbert}, choose a rational Hilbert
specialization and retain the resulting finite root extension.
At any finite time there are only finitely many retained
extensions. Apply \cref{lem:ordinary-scalar-lock} to all affected
finite-chart fields and surviving quotient fields, excluding also
the fixed nodes \(c_j\). Choose \(a\) in
\eqref{eq:ordinary-full-arrow} last, after the old fields, digits,
and \(H\), hence \(\rho\), are fixed. Formula
\eqref{eq:ordinary-kummer} is irreducible over every retained
extension, so preserves all of them by linear disjointness.
Carry each retained extension forward by base change after
every relevant arrow.
The factor \(p-a\) is a unit at every branch and a nonzero
scalar in every boundary. It changes none of the carrier,
divisibility, or nonvanishing requirements.

A lower-prefix requirement is unaffected by higher-row
refinements. For a final new prime, none of its requirements
is discarded, and the preceding quotient-field argument applies
to every subsequent relevant refinement. An inherited prime
has exactly its earlier quotient field by
\eqref{eq:ordinary-inherited}. Thus induction on prefixes
handles all height-one generic primes. The PID conclusion proved
below makes these all the nonzero generic primes.
A polynomial irreducible over a
union is irreducible over every earlier field containing its
coefficients; its rational specialization is eventually
prescribed and remains irreducible by the retained extension.
This proves rational Hilbertianity of all required union fields.
Fresh coordinates at an infinite cap give purely transcendental
extensions and automatically preserve the same finite extensions.

\subsection{The finite-cap generic PID}

For this subsection all current cap widths are fixed and finite.
Fix row \(i\) and branch \(j\), and put
\(F=\operatorname{Frac}C_{i-1}\). Let \(S\) be the field generated
over \(F\) by all silent coordinates in this row.
Here \(F A\), for a \(C_{i-1}\)-algebra \(A\) contained in the
constructed field, denotes \((C_{i-1}\setminus\{0\})^{-1}A\).
For any valuation centered at branch \(j\), with positive
\(\pi_j\), \eqref{eq:ordinary-silent} gives
\[
 f/g=(a/b)\frac{1+\pi_jh}{1+\pi_jh'}.
\]
Thus its restriction to \(S\) is immediate over its restriction
to \(F\): values agree, and residues agree when their value is
zero. In a valued quotient with zero lower contraction, no
nonzero silent polynomial vanishes, by the same identity.
Hence \(S\) embeds there and the same immediacy holds.
We use this assertion only before the first nonzero lower
contraction.

At an inactive branch every new block is silent, so its relative
generic ring is a field. At an active branch let
\(\mathfrak h\) be the inverse image of its new first conductor.
Equation \eqref{eq:ordinary-denominator} gives
\begin{equation}\label{eq:ordinary-first-conductor}
 F(C_i)_{\mathfrak m_{i,j}}
       =(C_i)_{\mathfrak h}[1/\pi_j].
\end{equation}
Indeed each denominator outside \(\mathfrak h\) is a lower
nonzero element times a branch unit, and every lower nonzero
element is a power of \(\pi_j\) times an element outside
\(\mathfrak h\).

The relative fused valuation on \(S\) is immediate over the
lower DVR \(C_{i-1,(\pi_j)}\); let its valuation ring be \(W\).
It is a DVR with the same residue field \(L\).
Put \(V=C_{i-1,(\pi_j)}\). Successively matching residues in
\(V\), subtracting, and dividing by the common uniformizer
\(\pi_j\) approximates every element of \(W\) modulo
\(\pi_j^N W\) by an element of \(V\), for every \(N\).
Thus \(V\) is dense in \(W\), and their completions canonically agree.
The whole \(W\) lies in \((C_i)_{\mathfrak h}\), since the
displayed silent ratio reduces membership to a lower DVR
fraction times units. Moreover
\begin{equation}\label{eq:ordinary-W-chart}
 (C_i)_{\mathfrak h}
       =\bigcup_s W[Y_{s,j}]_{(\pi_j,Y_{s,j})}.
\end{equation}
All silent coordinates and lower DVR coefficients lie in \(W\).
Every finite denominator outside the conductor has nonzero
constant residue, and conversely every such polynomial
denominator is outside the conductor. These facts prove both
inclusions. The active tuple has \(1+t_{j,i}\) independent
coordinates over \(S\). Independence over the whole silent
field follows by placing the finitely many coefficients of a
putative relation in one later separated block chart.

We prove that the generic ring of \eqref{eq:ordinary-W-chart}
has dimension at most one. Write \(t=t_{j,i}\). For \(t=0\)
its one-variable charts suffice. Otherwise a strict two-step
generic prime chain has persistent witnesses in all later
charts. The top quotient field \(\Lambda\) therefore satisfies
\(\operatorname{trdeg}_{S}\Lambda\le t-1\), by the polynomial
height formula. Dominate its local quotient at \(\mathfrak h\)
by a valuation \(v\), restricting to the DVR \(W\).
Put \(\delta=v(\pi_j)>0\) and
\[
 \Delta=\{\gamma:|\gamma|\le N\delta\text{ for some }N\},\qquad
 \Delta^-=\{\gamma\in\Delta:n|\gamma|<\delta\text{ for all }n\ge1\}.
\]
Coarsen by \(\Delta\) and pass to the residue field. The field
\(S\) embeds, since its nonzero elements have values in
\(\mathbb Z\delta\); \cref{lem:residue-extension} shows that
relative transcendence degree does not increase. Coarsen the
residual valuation by \(\Delta^-\). Its group
\(\Delta/\Delta^-\) is Archimedean, so it is rank one, normalized
by \(v(\pi_j)=1\). Its center still contains \(\pi_j\).
The only possibilities are \((\pi_j)\) and \(\mathfrak h\);
the killed nonzero generic prime cannot lie below the
height-one UFD prime \((\pi_j)\). The center is therefore
\(\mathfrak h\), and the resulting residue field has
transcendence degree at most \(t-1\) over \(L\).

If the nonzero images of all active coordinates in all later
charts have a common positive lower bound \(\epsilon\), put
\(\eta=\min(\epsilon,1)\). The silent identity forces the quotient
valuation to restrict to the same valuation on \(S\). Its separated
completion therefore contains the common completion of \(W\)
and \(V\), with the same coefficient embedding as the fused state.
Both coordinate states have values at least \(\eta\).
Every nonconstant monomial in an active arrow has degree at least
two and coefficients of nonnegative value; its difference at the
two states is a sum of terms, each with one difference factor and
at least one further factor of value at least \(\eta\).
Each full backward sweep
therefore raises the minimum difference value by at least
\(\eta\). Arbitrarily many sweeps force equality in the separated
rank-one completion. A killed nonzero polynomial witness then
contradicts the algebraic independence of the fused state over \(S\).
Otherwise choose a coordinate lift \(u_1\) whose value is
\(c\), \(0<c<1/4\), and whose boundary class is \(e\in M\).
For every \(\lambda\in L(s_1,\ldots,s_t)\), the element
\(\lambda e\) belongs to \(M\); choose a lift \(u_\lambda\).
Modulo \(\pi_j\),
\[
 u_\lambda u_{\lambda^{-1}}\equiv u_1^2,\qquad
 u_\lambda u_\mu\equiv u_{\lambda\mu}u_1,\qquad
 u_\lambda+u_\mu\equiv u_{\lambda+\mu}.
\]
All nonzero lifts have positive value. The first congruence
gives \(0<v(u_\lambda)<2c\); the second gives
\(v(u_\lambda)+v(u_\mu)=v(u_{\lambda\mu})+c\).
The homomorphism \(\lambda\mapsto v(u_\lambda)-c\) must vanish,
since \(v(u_{\lambda^n})>0\) for every positive and negative
integer \(n\). After division by \(u_1\) or \(u_1^2\), the
errors have positive value. The residues of
\(u_\lambda/u_1\) consequently embed \(L(s_1,\ldots,s_t)\)
in the residue field. This is the identity on \(L\): for
\(\lambda\in L\), use a lift from \(W\) times \(u_1\).
It contradicts the bound \(t-1\). Both cases are impossible.

Inductively \(C_{i-1}[1/P]\) is a PID. By
\eqref{eq:ordinary-inherited}, a generic prime with nonzero
lower contraction is principal and height one. A two-step
chain with zero lower contraction would remain such a chain
at a branch containing its top member and violate
\eqref{eq:ordinary-first-conductor} and the preceding argument.
Hence all nonzero primes of \(C_i[1/P]\) have height one.
A one-dimensional UFD is a PID: its maximal ideals are
principal, so a finite set of elements with gcd one generates
the unit ideal; choosing finitely many elements attaining the
minimum exponents among the prime factors of a fixed element
of an arbitrary ideal proves that ideal principal.

This finite-cap construction starts at any prescribed finite
nested stage and retains every earlier map, digit, and finite
extension requirement. All the additional tasks above concern
only future finite blocks.

\subsection{Infinite caps}

At an infinite cap adjoin independent coefficient variables
fairly, always with finite current width. Treat every boundary
target and every polynomial from every newly created stage.
To obtain generic PIDs, also treat every pair in every generic
prefix. Freeze the current finite widths and complete the
exact current nested stage by the finite-cap construction just
proved, retaining its current finite extension requirements.
In that auxiliary completion the pair has witnesses
\[
       a=ga',\qquad b=gb',\qquad g=ua+vb.
\]
All five elements and their equalities occur in a finite later
stage. Adopt just that finite block and discard the unused
continuation. The boundary valued fields were constructed with
their arrows, so this block extends the exact existing charts;
new independent coefficient variables can still be added
afterwards. No completed auxiliary ring is identified with
the eventual ring. The displayed identities, all silent
identities, and all prior nonvanishing and field-extension
conditions persist.

A diagonal ordering treats pairs at every prefix together with
all preceding requirements. The resulting generic prefixes are
Bezout UFDs, hence PIDs. For completeness, in a Bezout UFD,
successively taking a gcd with an element outside the currently
generated principal ideal strictly removes an irreducible
factor of a fixed initial element; this process terminates
and proves that every nonzero ideal is principal. All
previous boundary, factoriality, valuation, and arithmetic
conclusions thus hold for arbitrary infinite caps.

\subsection{The individual dimension bounds}

Fix a branch \(j\), put \(d=d_j\), \(t_i=t_{j,i}\), and
localize at its maximal ideal. The primes containing \(\pi_j\)
are exactly
\[
 p_0=(\pi_j)\subsetneq p_1\subsetneq\cdots
          \subsetneq p_d=\mathfrak m_j,\qquad
 p_s/\pi_jC=M_{j,d-s+1}+\cdots+M_{j,d}.
\]
Every other nonzero prime is generic and height one, since its
generic ring is a localization of the PID \(C_D[1/P]\).
Thus a base flag consists of zero, optionally one generic
prime, and a subsequence of this special flag.

Write \(F_i=\operatorname{Frac}C_i\).
For a finite active interval, an active row contributes at
most its width \(1+t_i\) to residue transcendence. The
silent extension is immediate by \eqref{eq:ordinary-silent},
and an inactive row contributes zero. These assertions apply
over the full lower field, regardless of its transcendence
degree: every alleged finite relation or finite tuple occurs
in a common later chart. At a first nonzero prime contraction
the silent field still embeds, and a persistent nonzero
polynomial witness in the active chart lowers its quotient
transcendence degree by one.

For the first special prime \(p_s\), put \(\ell=d-s\).
Its contraction to \(C_\ell\) is \((\pi_j)\), with residue
field \(L_{j,\ell}=\kappa(p_s)\). Every valuation centered
there restricts on \(F_\ell\) to its \(\pi_j\)-DVR.
Iterating immediacy and \cref{lem:residue-extension} through
the active rows \(\ell+1,\ldots,d\), and through subsequent
inactive rows, gives
\begin{equation}\label{eq:ordinary-first-bound}
 \tau_C(0,p_s)\le s+\sum_{i=\ell+1}^{d}t_i.
\end{equation}
Suppose instead that a nonzero generic prime \(q\) precedes
\(p_s\). Its contraction to \(C_\ell\) is zero, since it is
contained in \((\pi_j)\) there and avoids \(\pi_j\).
Let \(h>\ell\) be the first row at which its contraction is
nonzero. This row is active: at an inactive row
\eqref{eq:ordinary-silent} rules out a new nonzero prime
over zero. Up to row \(h-1\) the residue contributions are
at most \(1+t_i\); at row \(h\) the persistent quotient
witness gives at most \(t_h\). Above \(h\),
\eqref{eq:ordinary-inherited} identifies the quotient fields
with the preceding lower quotient, so they do not grow.
Consequently
\begin{equation}\label{eq:ordinary-generic-bound}
 \tau_C(q,p_s)
   \le\sum_{i=\ell+1}^{h-1}(1+t_i)+t_h
   \le s-1+\sum_{i=\ell+1}^{d}t_i,\qquad
 \tau_C(0,q)=0.
\end{equation}
Finally, for a special jump \(p_a\subsetneq p_b\), the local
quotient has fraction field \(L_{j,d-a}\), over the lower
field \(L_{j,d-b}=\kappa(p_b)\). Its dominating valuation
is trivial on the lower field and nontrivial on the upper
one. A positive element together with independent residue
lifts is algebraically independent over the lower field.
Equation \eqref{eq:ordinary-boundary-td} therefore gives
\begin{equation}\label{eq:ordinary-special-bound}
 \tau_C(p_a,p_b)\le b-a-1+
                 \sum_{i=d-b+1}^{d-a}t_i.
\end{equation}

The estimates just proved were needed only for intervals
whose own caps are finite. If an interval contains an
infinite cap, its sum of truncated caps already contains
\(n\), so the bound \(\min(n,\tau)\le n\) suffices.
Using
\(\min(n,c+\sum t_i)\le c+\sum\min(n,t_i)\), the contributions
of a first special edge, a pair of edges through a generic
prime, and a later special jump are respectively bounded by
\[
 s+1+\sum_{i=d-s+1}^{d}\min(n,t_i),\qquad
 s+1+\sum_{i=d-s+1}^{d}\min(n,t_i),\qquad
 b-a+\sum_{i=d-b+1}^{d-a}\min(n,t_i).
\]
Their intervals are disjoint along a flag. Enlarging a
terminal interval if necessary, \cref{thm:capacity} yields
\[
 \dim C_j[n]\le n+d+1+\sum_{i=1}^{d}\min(n,t_i).
\]
The boundary chains in \eqref{eq:ordinary-boundary-lower},
lifted and preceded by zero below \(\pi_jC_j[n]\), give
the reverse inequality. They also force equality in
\eqref{eq:ordinary-boundary-lower}, proving
\eqref{eq:ordinary-realization}.

Each \(C_j[1/\pi_j]\) is a localization of \(C_D[1/P]\).
Its nonzero prime quotient field is the corresponding global
generic prime quotient field: a nonzero localization of a
field does not change it. The arithmetic construction
therefore supplies exactly the fields asserted in
\cref{thm:ordinary}, while scalar fusion supplies the stated
valuations. This completes its proof.

\section{A relative Hilbert construction}
\label{sec:relative}

This section gives the construction which supplies the delayed basic
profiles.  All rings contain \(\mathbb Q\).  We use the rational
Hilbert property and its classical-Hilbert-irreducibility proof from
Lemma~\ref{lem:ordinary-rational-hilbert}.

\begin{lemma}[Finite Hilbert fibers]\label{lem:relative-hit}
Let \(K\) be rationally Hilbertian.  If an integral affine
\(K\)-algebra \(E\) is finite over a polynomial subring
\(K[h_1,\ldots,h_r]\), then there are \(c_i\in\mathbb Q\) such that
\(E/(h_i-\delta_i c_i)_i\) is a finite field extension of \(K\),
where the \(\delta_i\in K^\times\) may be specified in advance.
The tuple can avoid any prescribed nonzero polynomial conditions.
\end{lemma}

\begin{proof}
Choose a primitive element \(\theta\in E\) for
\(\operatorname{Frac}E/K(h)\); a \(K\)-linear combination of
algebra generators suffices.  Express those generators in powers
of \(\theta\) and invert a nonzero polynomial \(b(h)\) clearing
the coefficients.  Then \(E[1/b]=K[h,1/b][\theta]\).
Avoid \(b\) and the discriminant, and use
the defining Hilbert property on the primitive polynomial after the
invertible substitution \(h_i=\delta_iT_i\).  The resulting fiber
is exactly its finite field.
\end{proof}

\subsection{Two algebraic modification lemmas}

\begin{lemma}[High-power normalization]\label{lem:relative-normalization}
Let \(K\) be a field of characteristic zero, and let
\(0\ne q\in K[x,z_1,\ldots,z_k]\).  Choose distinct primes
\(\ell_i>\deg_{z_i}q\), put \(N_i=\ell_iM_i+1\), and set
\(h_i=z_i^{\ell_i}-x^{N_i}\), where the \(M_i\ge1\) are arbitrary.
Then \(K[x,z]/(q)\) is finite over the algebra generated by the
\(h_i\).  The assertion is uniform for every nonzero specialization
of the coefficients of \(q\).

Suppose, in addition, that an element \(f\) satisfies
\begin{equation}\label{eq:relative-weight-bound}
f^d+\sum_{j=1}^d c_j(x,z)f^{d-j}=0,
 \qquad
 \sum_i\alpha_i/\ell_i<j
 \quad\text{for every monomial }z^\alpha\text{ in }c_j.
\end{equation}
If \(z_i^{\ell_i}=b_i(x)+fT_i\), then the algebra generated by
\(f,z_1,\ldots,z_k\) is finite over the coefficient ring generated
by \(x,T_1,\ldots,T_k\).  These statements hold with any fixed
coefficient ring in the second assertion.
\end{lemma}

\begin{proof}
Give \(x\) weight one and \(z_i\) weight \(M_i+1/\ell_i\).
Distinct monomials of \(q\) have distinct weights: multiply an
equality of weights by \(\prod_i\ell_i\), reduce modulo each
\(\ell_i\), and use that the difference of the corresponding
\(z_i\)-exponents has absolute value less than \(\ell_i\).
Thus there is a unique largest-weight monomial
\(c x^{a_0}\prod_i z_i^{a_i}\).

The algebra
\[
 K[H,x,z]/(z_i^{\ell_i}-x^{N_i}-H_i)_i
\]
is free of rank \(b=\prod_i\ell_i\) over \(K[H,x]\).  The norm
of \(q\) in this free algebra has a nonzero constant leading
coefficient as a polynomial in \(x\).  To verify this explicitly,
put \(x=s^{-b}\) over an algebraic closure of \(K(H)\).  Its
simultaneous roots have the expansions
\[
 z_i=\zeta_i s^{-bN_i/\ell_i}
          (1+H_i s^{bN_i})^{1/\ell_i},\qquad\zeta_i^{\ell_i}=1.
\]
At each root the leading term of \(q\) is its unique largest-weight
monomial.  Their product has \(x\)-degree
\(b a_0+\sum_i a_iN_ib/\ell_i\), and leading coefficient
\(\pm c^b\).  The adjugate identity puts this norm in the ideal
generated by \(q\).  Dividing by its leading coefficient gives
a monic equation for \(x\) modulo \(q\); the displayed monic
equations then make all \(z_i\) integral.  This proves finiteness.
A specialization only deletes monomials, so the same weight argument
works for every nonzero specialized polynomial.

For the last assertion give \(f\) weight one, \(z_i\) weight
\(1/\ell_i\), and the coefficient ring weight zero.  The monic
equation for \(f\) reduces \(f^d\) to terms of strictly lower
weight.  Replacing \(z_i^{\ell_i}\) by \(b_i(x)+fT_i\) does
not increase weight, and its equal-weight term strictly lowers total
\(z\)-degree.  A common integral multiple of weight, followed by
total \(z\)-degree, is a decreasing pair of nonnegative integers.
The reductions terminate and leave the finite spanning set
\(f^a\prod_i z_i^{a_i}\), with \(a<d\) and \(a_i<\ell_i\).
\end{proof}

\begin{lemma}[Uniform horizontal modification]\label{lem:relative-horizontal}
Let \(B\) be a PID containing \(\mathbb Q\), allowing a field,
and put \(F=\operatorname{Frac}B\).  Suppose \(E\) is a UFD
finite locally free over \(B[x,z_1,\ldots,z_k]\), \(k\ge1\),
and \(E/aE\) is a domain for each prime element \(a\in B\).
Let \(f\) be a prime of \(E\) not associated with a lower prime.
Choose the distinct \(\ell_i\) larger than the coordinate degrees
of the norm of \(f\), and satisfying
\eqref{eq:relative-weight-bound} for its characteristic polynomial.
Put \(g_i=z_i^{\ell_i}-x^{\ell_iM_i+1}-b_i\), \(b_i\in B\).
If
\[
 E_F/(f,g_1,\ldots,g_k)=K'
\]
is a finite field extension of \(F\), then
\[
 E'=E[T_1,\ldots,T_k]/(fT_i-g_i)_i
\]
is a UFD with the same fraction field as \(E\), finite locally
free over \(B[x,T]\).  Every \(E'/aE'\) is a nonzero domain
containing \(E/aE\), and \(E'/fE'\) is a domain, torsionfree
over \(B\), with generic fiber \(K'[T]\).
\end{lemma}

\begin{proof}
Write \(Q_f(U)=U^d+c_1U^{d-1}+\cdots+c_d\) for multiplication
by \(f\) on the finite locally free module \(E\), and let
\(N_f=\pm c_d\).  For every lower prime \(a\), the element
\(\bar f\) is nonzero in the domain \(E/aE\).  Its multiplication
map becomes injective on a finite-dimensional vector space over
\(\operatorname{Frac}((B/a)[x,z])\); hence
\(\overline{N_f}\ne0\).  The norm commutes with reduction by
local freeness.  Lemma~\ref{lem:relative-normalization} applied to
this nonzero polynomial, whose norm lies in \((\bar f)\) by the
characteristic equation, shows that
\[
 E/(a,f)\text{ is finite over }(B/a)[h_1,\ldots,h_k],
 \quad h_i=z_i^{\ell_i}-x^{\ell_iM_i+1}.
\]
Thus \(E/(a,f,g)\) is finite-dimensional over \(B/a\), possibly
zero.  At every prime in its support, \(a,f,g_1,\ldots,g_k\)
is a system of parameters of a Cohen--Macaulay local ring of
dimension \(k+2\).  Here \(E\) is Cohen--Macaulay because it
is finite flat over a regular ring, and its local dimensions equal
those of that ring by going down and incomparability.  Over \(F\),
the hypothesis similarly makes \(f,g_1,\ldots,g_k\) a parameter
sequence at its support.

Put \(r_i=fT_i-g_i\).  Inspect the presentation locally at a
prime containing \(f\); it then contains all \(g_i\).  If its
contraction to \(B\) is \((a)\), the preceding parameter sequence
extends to the polynomial local ring, and replacing \(g_i\) by
\(r_i\) leaves \(a,f,r_1,\ldots,r_k\) regular.  Permutability
of regular sequences in a Noetherian local ring shows that both
\((a,f)\) and \((f,a)\) are regular after quotienting by the
\(r_i\).  For contraction zero use \(f,g\) over \(F\).
Where \(f\) is invertible the equations eliminate the \(T_i\).
Consequently \(f\) is regular on \(E'\) and on every \(E'/aE'\),
each \(a\) is regular on \(E'/fE'\), and \(E'\) is
Cohen--Macaulay.  These are assertions at the common support;
no global parameter-sequence inference is needed.

It follows that
\[
 E'\hookrightarrow E'[1/f]=E[1/f],\qquad
 E'/aE'\hookrightarrow(E/aE)[1/\bar f].
\]
The second target is nonzero and contains \(E/aE\) injectively.
The quotient \(E'/fE'\), being torsionfree over the PID \(B\),
embeds in \(K'[T]\).  These give all domain assertions and
primality of \(f\).  Nagata's criterion gives factoriality.

The last assertion of Lemma~\ref{lem:relative-normalization},
followed by the original finiteness of \(E\), makes \(E'\)
finite over \(B[x,T]\).  Its unchanged fraction field has
relative transcendence degree \(k+1\), so \(x,T\) are algebraically
independent.  Put \(R=B[x,T]\).  For each \(\mathfrak p\in\operatorname{Spec}R\),
the finite \(R_{\mathfrak p}\)-algebra
\(S=E'\otimes_R R_{\mathfrak p}\) is semilocal.  Going down and
incomparability give \(\dim S_{\mathfrak q}=\dim R_{\mathfrak p}\)
at every maximal ideal \(\mathfrak q\) of \(S\).  A regular system
of parameters of \(R_{\mathfrak p}\) is therefore a parameter
sequence in each Cohen--Macaulay ring \(S_{\mathfrak q}\), hence
is regular on \(S\).  Thus \(S\) is a finite maximal Cohen--Macaulay
\(R_{\mathfrak p}\)-module, and the Auslander--Buchsbaum formula
makes it free.  This proves local freeness.
\end{proof}

\begin{lemma}[Saturation and factoriality]\label{lem:relative-saturation}
Let \(R\) be a domain, \(0\ne\pi\in R\), and
\(r_1,\ldots,r_k\in R[T]\).  If their reductions form a
regular sequence modulo \(\pi\), then
\[
 ((r_1,\ldots,r_k):\pi)=(r_1,\ldots,r_k).
\]
If its localization at \(\pi\) is a nonzero domain, the quotient
is a domain with the displayed exact special fiber.  No
Noetherian hypothesis is needed.  Also, an atomic domain with a
prime \(\pi\) and factorial localization at \(\pi\) is a UFD.
\end{lemma}

\begin{proof}
Reduce \(\pi H=\sum_i u_i r_i\) modulo \(\pi\).  The first
Koszul homology of a regular sequence vanishes over any ring
(induction on its length proves this by separating the last
coefficient).  The relation among the \(\bar r_i\) is therefore
a sum of the antisymmetric Koszul relations.  Subtract their lifts
from \((u_i)\); all remaining coefficients are divisible by
\(\pi\).  Cancel \(\pi\) in the polynomial domain.  This proves
saturation and the asserted injection into the localization.

For the last assertion an irreducible other than \(\pi\) remains
irreducible after inverting \(\pi\): clear denominators and
cancel prime factors \(\pi\) in a putative factorization.  It is
prime there.  Its principal ideal contracts unchanged, by the same
cancellation, so it was prime already.  Atomicity finishes the proof.
\end{proof}

\subsection{The relative row}

\begin{theorem}\label{thm:relative}
Let \(C\) be a countable finite-dimensional local UFD containing
\(\mathbb Q\), let \(\pi\) be prime, and put
\[
 D=C/\pi C,\quad B=C[1/\pi],\quad
 F=\operatorname{Frac}C,\quad L=\operatorname{Frac}D.
\]
Assume that \(B\) is a PID, that \(F\) and every \(B/aB\)
for a prime \(a\in B\) are rationally Hilbertian, and that a
discrete valuation \(\nu\) dominates \(C\), with
\(\nu(\pi)=1\).  Assume also
\begin{equation}\label{eq:relative-lower-profile}
\dim C[n]=1+\dim D[n]\qquad(n\ge0).
\end{equation}
For every \(t\in\mathbb N_{>0}\cup\{\infty\}\), there is
a countable local UFD \(A\supset C\), with the same residue
field and \(\pi\) prime, such that \(A[1/\pi]\) is a PID and
\begin{equation}\label{eq:relative-profile}
\dim A[n]=\dim C[n]+\max\{1,\min(n,t)\}\quad(n\ge0).
\end{equation}
Its fraction field is a rational function field in \(t\)
variables over \(F\), with countably many variables when
\(t=\infty\).  In particular the lower rings supplied by
Theorem~\ref{thm:ordinary} satisfy the required hypotheses.
\end{theorem}

\begin{proof}[Construction and proof]
First let \(t<\infty\), put \(k=t-1\), and work inside the
fixed field \(F^*=F(x,z_1,\ldots,z_k)\).  Extend \(\nu\)
by the Gauss rule giving all new variables value one.  Also
extend the valuation of the DVR \(C_{(\pi)}\) by that rule,
calling it \(w\).  Then
\begin{equation}\label{eq:relative-gauss}
\kappa(w)=L(x/\pi,z_1/\pi,\ldots,z_k/\pi),
 \qquad w(\pi)=w(x)=w(z_i)=1.
\end{equation}
Its restriction is zero on \(C\setminus\pi C\).  Thus \(w\)
will be centered at the upper augmentation prime, whereas
\(\nu\) dominates the full local ring.

Maintain a full local stage \(S\subset F^*\), and an affine
generic model \(E\), with these invariants:
\begin{enumerate}[label=(\roman*)]
\item \(S\) is a local UFD containing \(C\) locally, with the
same residue field; all top generators are positive for both
valuations and \(\nu\) dominates \(S\).
\item \(E\) is a UFD finite locally free over
\(B[x_s,z_{s,1},\ldots,z_{s,k}]\); every lower prime fiber
\(E/aE\) is a domain; and \(S[1/\pi]\) is a localization of \(E\).
\item The special fiber is the local ring at the lower maximal
ideal and top origin of \(\Gamma_s[z_{s,1},\ldots,z_{s,k}]\).
Here \(\Gamma_s\) is a domain finite free over \(D[x_s]\),
augmented onto \(D\), its augmentation ideal is nilpotent
modulo \(x_s\), and
\[
 \operatorname{Frac}(L\Gamma_s)=L(u_s),\qquad
 x_s=\lambda_su_s^{d_s},\quad \lambda_s\in L^\times, d_s>0.
\]
\end{enumerate}
Initially these are the local polynomial ring \(C[x,z]\), the
affine ring \(B[x,z]\), and \(\Gamma_0=D[x]\).  Polynomial
rings over UFDs are factorial without a Noetherian hypothesis.
The generic models are Noetherian because \(B\) is a PID; the
full stages need not be Noetherian.

\smallskip\noindent\emph{Horizontal actions.}
For \(k>0\), choose a currently horizontal irreducible in
\(S[1/\pi]\), represented by a nonunit of \(S\).  The affine
UFD \(E\) supplies an associated affine prime \(f\), not
associated with a lower prime.  Its coefficients lie in
\(C[1/\pi]\), so multiply it by a power of \(\pi\) to put
it in \(S\), and divide out its largest \(\pi\)-power there.
It stays an affine prime of \(E\), belongs to \(S\), is not
\(\pi\)-divisible, and has positive \(\nu\)-value.

Choose distinct primes \(\ell_i\) satisfying all bounds in
Lemma~\ref{lem:relative-horizontal}, exceeding the degrees in the
passive variables of a numerator of \(\bar f\), and not dividing
\(d_s\).  Increase them, then choose \(N_i=\ell_iM_i+1\)
and integers \(q_i\) so that, for \(v=\nu,w\),
\begin{equation}\label{eq:relative-positive}
N_iv(x_s)>\ell_i v(z_{s,i}),\quad
 q_i>\ell_i v(z_{s,i})>v(\pi f).
\end{equation}
There are arbitrarily large permissible primes.  The norm
normalization proves that \(E_F/fE_F\) is finite over the
polynomial algebra in \(z_{s,i}^{\ell_i}-x_s^{N_i}\).
Lemma~\ref{lem:relative-hit} therefore chooses
\(c_i\in\mathbb Q^\times\) giving a finite field fiber, with
the prescribed factors \(\pi^{q_i}\).  Adjoin the actual fractions
\begin{equation}\label{eq:relative-horizontal-fractions}
Z_i=\frac{z_{s,i}^{\ell_i}-x_s^{N_i}-c_i\pi^{q_i}}{\pi f}
\end{equation}
and localize at the old maximal ideal and the new coordinates.
Lemma~\ref{lem:relative-horizontal}, with \(\pi f\) in place
of the associate \(f\), gives the new generic affine invariants.

We justify the full presentation, where \(C\) may be
non-Noetherian.  Its reduced equations are the successive monic
equations \(z_{s,i}^{\ell_i}=x_s^{N_i}\).  Each quotient is a
domain: over \(L(u_s)\), the exponent \(d_sN_i\) is coprime
to \(\ell_i\), and a valuation at \(u_s=0\) proves
irreducibility.  At later equations the old exponent is multiplied
by the preceding distinct primes, so the same coprimality holds.
Monic freeness gives the injection before extending coefficients
to \(L\).  In particular the equations are a regular sequence.
The bound on the passive degrees ensures that \(\bar f\) remains
nonzero: its numerator is already a nonzero standard monomial
combination for the successive monic equations, and its local
denominator still has unit augmentation.

Lemma~\ref{lem:relative-saturation} makes the full presentation
\(\pi\)-torsionfree.  Its generic localization is a localization
of the affine domain just proved.  The fractions
\eqref{eq:relative-horizontal-fractions} have positive value by
\eqref{eq:relative-positive}; every intended inverted denominator
therefore evaluates to a nonzero \(\nu\)-unit.  This proves the
localization is nonzero and the full presentation is its asserted
subring of \(F^*\), with exactly the displayed domain as special
fiber.  Its elements have bounded factorization length by the
integer \(\nu\)-value.  It is atomic, and the last part of
Lemma~\ref{lem:relative-saturation} makes it a UFD.

The new curve algebra remains finite free over \(D[x_s]\) and
its augmentation ideal is nilpotent modulo \(x_s\).  Rationality
of its generic field does not require roots of coefficient units:
if \(z^\ell=\lambda^N u^{dN}\) and
\(a\ell+b dN=1\), put \(v=u^a z^b\).  Then
\(u=\lambda^{-Nb}v^\ell\) and
\(z=\lambda^{Na}v^{dN}\).  Iteration supplies the asserted new
rational parameter.  The old generic quotient by \(f\), after
inverting the nonzero lower coefficients, has image in the finite
field chosen above.

\smallskip\noindent\emph{Actions at a lower prime.}
Let \(a\) be a prime of \(B\), represented by a prime of \(C\)
different from \(\pi\).  It is prime in \(E\) by (ii).
Choose distinct large prime exponents \(m_i\), coprime to the
current curve exponent, put \(N_i=m_iM_i+1\), fix \(e\ge1\),
and choose \(\ell\) and all \(q_i\) sufficiently large.  Set
\begin{equation}\label{eq:relative-mixed}
\begin{gathered}
g_0=x_s^\ell-c_0\pi^{q_0},\quad
 g_i=z_{s,i}^{m_i}-x_s^{N_i}-c_i\pi^{q_i},\\
 X=g_0/a^e,\qquad Z_i=g_i/(\pi a).
\end{gathered}
\end{equation}
Choose the exponents so that each first displayed monomial has
strictly smaller value than the other terms and strictly larger
value than its denominator, for both \(\nu\) and \(w\).
This is possible because \(e\) is fixed and the old coordinates
are positive.  The functions
\(x_s^\ell,z_{s,i}^{m_i}-x_s^{N_i}\) form a finite polynomial
normalization of \(E/aE\): the old \(x_s,z_{s,i}\) are integral
over them by their monic equations.  The Hilbert property of
\(B/aB\) chooses rational nonzero \(c_i\) with
\begin{equation}\label{eq:relative-mixed-field}
E/(a,g_0,g_1,\ldots,g_k)\text{ a finite field over }B/aB.
\end{equation}

For clarity, the affine presentation of this mixed action has no
hidden torsion.  At a prime containing \(a\), condition
\eqref{eq:relative-mixed-field} makes
\(a,g_0,\ldots,g_k\) a parameter sequence in the old
Cohen--Macaulay local ring.  In the polynomial local ring replace
the \(g_i\) by \(a^eX-g_0,\pi aZ_i-g_i\); after \(a\)
these are the same regular sequence up to sign.  Permuting it
makes \(a\) regular on their quotient.  Away from \(a\) the
equations eliminate the new coordinates.  The quotient thus
embeds in \(E[1/a]\), is a domain, and is Cohen--Macaulay.  Its
quotient by \(a\) is a polynomial ring over the field in
\eqref{eq:relative-mixed-field}, so \(a\) is prime.  Every other
lower prime \(b\) has \(a\) invertible in \(B/bB\), so the
new \(b\)-fiber is just \(E/bE\) with the fractions evaluated;
it is a domain.  Nagata gives a UFD.  The equations
\[
 x_s^\ell=a^eX+c_0\pi^{q_0},\qquad
 z_{s,i}^{m_i}=x_s^{N_i}+c_i\pi^{q_i}+\pi aZ_i
\]
make it finite over \(B[X,Z]\).  The unchanged fraction field
proves these new base coordinates independent.  The
Cohen--Macaulay and regular-base argument in
Lemma~\ref{lem:relative-horizontal} proves local freeness.

Its exact full special fiber also needs a non-Noetherian check.
First impose the monic equations on the old passive block, giving
a domain \(\Gamma_1\) free and finite over \(D[x_s]\).  The
remaining equation is \((\bar a)^eX=x_s^\ell\).  The quotient
\(\Gamma_1[X]/((\bar a)^eX-x_s^\ell)\) is
\(\bar a\)-torsionfree.  Indeed, reduce
\(\bar a H=((\bar a)^eX-x_s^\ell)G\) modulo \(\bar a\).
Multiplication by \(x_s^\ell\) is injective on
\(\Gamma_1/\bar a\Gamma_1\), since \(\Gamma_1\) is free
over \(D[x_s]\).  Thus \(G\) is \(\bar a\)-divisible,
and cancellation proves saturation.  Inverting \(\bar a\)
gives a domain, so the quotient is a domain.  It is finite free
over \(D[X]\), using the monic equation for \(x_s\).  All old
augmentation elements are nilpotent modulo \(X\), and its
generic field is the same rational curve with
\(X=x_s^\ell/(\bar a)^e\).

The full reduced sequence is the successive monic equations
followed by this nonzero equation over a domain.  Apply
Lemma~\ref{lem:relative-saturation} to obtain \(\pi\)-saturation.
The actual positive fractions in \eqref{eq:relative-mixed} ensure
nonzero localization.  Atomicity and the generic UFD then prove
that the full stage is a UFD and \(\pi\) is prime.  A lower
prime cannot be lost on localizing: its value under \(\nu\) is
positive; if it became a unit after inverting \(\pi\), the UFD
would make it a \(\pi\)-power times a unit, contrary to its
nonzero special reduction.  The same reasoning retains the
chosen horizontal prime.  This completes preservation of all
stage invariants, including local survival of the affine fibers.

\smallskip\noindent\emph{Fair scheduling and factoriality of the union.}
Enumerate the elements of all countably many full stages and
repeat their tasks.  When an element is currently an irreducible
horizontal nonunit, perform its horizontal action; perform a
mixed action at each lower prime cofinally.  Each final
irreducible is irreducible in every later full stage containing
it, since nonunits remain nonunits under the dominating valuation.
Thus an eventual horizontal irreducible receives arbitrarily
late actions.  The integer valuation bounds every factorization
in the union \(A\), so it is atomic; an irreducible is prime
in all later full UFD stages, hence is prime in the union.
Therefore \(A\) is a local UFD.  The domain special fibers
and the preserved lower primes remain domains at the limit;
noninjective special-fiber transition maps cause no difficulty,
since products are witnessed at a finite stage.

Write \(G=A[1/\pi]\).  Repeated mixed actions put every old
element of \(G/aG\) into a finite field over \(B/aB\).
An old denominator has nonzero image in this field, since otherwise
the localized quotient would be zero and \(a\) would be a unit.
Thus this includes old localized elements.
Its nonzero elements keep their inverses under subsequent maps,
so \(G/aG\) is algebraic over that field and hence is itself
a field.  After inverting \(B\setminus\{0\}\), repeated
horizontal actions have the same consequence for every
horizontal prime quotient, now over \(F\).  Thus the latter
localization has dimension at most one.  If \(k=0\), this
conclusion requires no horizontal action: every nonzero prime
quotient of each affine generic curve over \(F\) is already
algebraic over \(F\), and a persistent nonzero prime witness
retains that property in every later chart.

A nonzero horizontal prime cannot lie below \(aG\), since
the latter is principal of height one in a UFD.  These facts
give \(\dim G\le1\).  A UFD of dimension at most one is a
PID: the gcd of a nonzero ideal is the gcd of finitely many
of its elements, since one fixed nonzero element has only
finitely many prime factors.  Divide the ideal by that gcd.
If the resulting ideal were proper, a maximal ideal above it
would be generated by a prime element (every nonzero prime
has height one), contradicting that the finite gcd is one.
Consequently a height-one prime \(q\) of \(A\) avoiding \(\pi\)
has quotient fraction field algebraic over \(F\) if \(q\cap C=0\),
or algebraic over \(B/aB=\operatorname{Frac}(C/aC)\) if its
contraction is the lower prime \((a)\).

\smallskip\noindent\emph{The boundary.}
Let \(\mathcal D=A/\pi A\), let
\(\epsilon:\mathcal D\to D\) be the augmentation, and put
\(M=\ker\epsilon\).  Every passive block is eventually
processed into a later curve.  To justify arbitrarily late
actions when \(k>0\), the nonzero old curve coordinate has
nonzero special image in every later curve and has positive
\(\nu\)-value.  A final irreducible factor other than \(\pi\)
is either horizontal or a lower prime; its repeated queue
therefore supplies such actions.  For \(k=0\) there is no
passive block to consume.

Every finite subset of \(L\mathcal D\) consequently lies in a
localization of an affine curve with rational fraction field
over \(L\).  Finite witnesses of strict prime inclusions show
that this ring has dimension at most one.  Its nonzero
augmentation prime has residue field \(L\), giving equality.
A persistent nonzero prime witness in a curve chart makes all
later residue fields algebraic over \(L\); hence every nonzero
prime of \(L\mathcal D\) has algebraic residue field.  Its
fraction field has transcendence degree one over \(L\): old
curve fields inject into the later curve fields, each finitely
many passive coordinates eventually becomes algebraic over an
old curve parameter, and that parameter stays transcendental.

If \(Q\) is a prime of \(\mathcal D\) contracting to
\(0\ne p\subset D\), lift a nonzero element of \(p\) to
\(C\) and factor it there.  Some prime factor \(a\ne\pi\)
has nonzero reduction in \(p\).  At a mixed \(a\)-action the
equation \((\bar a)^eX=x_s^\ell\) forces \(x_s\) into
\(Q\); nilpotence modulo \(x_s\) forces all earlier curve
augmentation elements into \(Q\).  Repeating after every
passive block has been consumed gives \(M\subset Q\).  Thus
\begin{equation}\label{eq:relative-boundary-primes}
Q=\epsilon^{-1}(p)\quad\text{if }Q\cap D=p\ne0.
\end{equation}

Here are the dimension details needed from these boundary facts.
For a nonzero zero-contraction boundary prime \(P\), the first
capacity \(\tau_{\mathcal D}(0,P)\) is zero: a centered
valuation is nontrivial and trivial on \(L\), and one positive
element together with independent residue lifts are algebraically
independent over \(L\).  A field of transcendence degree one
therefore leaves no residue transcendence.  Also
\(\operatorname{Frac}(\mathcal D/P)/L\) is algebraic.  Write
\(P_p=\epsilon^{-1}(p)\) for \(p\ne0\).  Capacities between
two such primes are exactly those of \(D\), because the
quotients are \(D/p\).  Lemma~\ref{lem:residue-extension} gives
\[
 \tau_{\mathcal D}(0,P_p)\le1+\tau_D(0,p),\qquad
 \tau_{\mathcal D}(P,P_p)\le\tau_D(0,p).
\]
A flag has at most one nonzero zero-contraction prime.  If it
jumps directly to \(P_p\), use
\(1+\min(n,1+\tau)\le2+\min(n,\tau)\); if it first
passes through \(P\), use its zero first capacity and the
second inequality.  All later contributions are those of a
flag of \(D\).  Theorem~\ref{thm:capacity} gives the upper
bound below.  For the lower bound lift a chain through the
surjection to \(D[n]\) and prepend zero below \(M[n]\).
Consequently
\begin{equation}\label{eq:relative-boundary-profile}
\dim\mathcal D[n]=1+\dim D[n].
\end{equation}

\smallskip\noindent\emph{Capacity and exact profile.}
Let \(h\) be the inverse image of \(M\) in \(A\).  Its
contraction to \(C\) is \((\pi)\), with residue field \(L\).
The valuation \(w\) dominates \(A_h\): a denominator outside
\(h\) has nonzero lower augmentation, hence value zero.
Equation~\eqref{eq:relative-gauss} gives \(\tau_A(0,h)\ge t\).
Conversely a valuation centered there restricts on \(F\) to
the DVR \(C_{(\pi)}\).  The residue extension inequality and
\(\operatorname{trdeg}_F F^*=t\) give the reverse inequality.
Thus \(\tau_A(0,h)=t\).  The prime \(\pi A\) has height
one, the generic ring is a PID, and the boundary localization
at \(M\) has dimension one, so \(\operatorname{ht}h=2\).

The prime \(\pi\) and the entire boundary give
\(\dim A[n]\ge\dim C[n]+1\).  Concatenating the edge
\(0<h\), of capacity \(t\), with a flag of \(A/h=D\)
gives \(\dim A[n]\ge\dim C[n]+\min(n,t)\).
We prove the reverse bound explicitly.  The prime classification
above makes \(A\) finite-dimensional: a flag contains at most
one nonzero prime avoiding \(\pi\), followed by a boundary
flag.  Apply Theorem~\ref{thm:capacity} to flags beginning at zero.

If the first nonzero prime is \(\pi A\), its first capacity
is zero since \(A_{(\pi)}\) is a DVR; the rest is bounded by
\eqref{eq:relative-boundary-profile}, giving \(\dim C[n]+1\).
If that first prime is \(q\) avoiding \(\pi\), it has height
one and first capacity zero.  For \(q\cap C=0\), delete this
one initial edge and contract the rest to \(C\).  The
contractions are strict: at most one nonzero generic-curve
prime lies over \((\pi)\), and higher primes are the distinct
augmentation preimages in \eqref{eq:relative-boundary-primes}.
The field \(\operatorname{Frac}(A/q)\) is algebraic over
\(F\); subsequent fields are lower quotient fields or
algebraic extensions of them.  The residue extension inequality
bounds every remaining capacity by its lower counterpart.  This
case contributes at most \(\dim C[n]+1\).  If instead
\(q\cap C=(a)\ne0\), all subsequent lower contractions
contain both \(a\) and \(\pi\), so are augmentation preimages.
Contract the whole flag to \(0<(a)<\cdots\); algebraicity of
the quotient field over \(\operatorname{Frac}(C/aC)\) bounds
all capacities without any extra edge.  The bound is \(\dim C[n]\).

Finally suppose the first nonzero prime \(Q\) properly contains
\(\pi\).  Its contraction \(c\) is nonzero, and contractions
of the whole flag are strict for the same boundary reason.
Only the first edge can gain capacity: restricting a centered
valuation from \(F^*\) to \(F\) gives
\[
 \tau_A(0,Q)\le t+\tau_C(0,c),\qquad
 \min(n,\tau_A(0,Q))
 \le\min(n,t)+\min(n,\tau_C(0,c)).
\]
The endpoint residue extension is algebraic for a generic-curve
prime, and identical for an augmentation preimage.  Every later
capacity is bounded by the corresponding contracted one.  The
flag lengths agree, giving the bound \(\dim C[n]+\min(n,t)\).
These cases prove \eqref{eq:relative-profile}.

\smallskip\noindent\emph{Infinite capacity.}
For \(t=\infty\), start with \(F(x,z_1,z_2,\ldots)\) and
the two Gauss valuations, still integer-valued.  Introduce one
fresh independent positive variable at each round; each stage
uses a finite tuple.  At every horizontal or mixed action process
the entire current tuple, with exponents chosen for that finite
stage.  The finite affine and saturation proofs apply literally.
Each old variable is consumed into a later boundary curve, and
each old generic quotient element is covered by a later action
at its final prime.  Thus the same boundary, factorization, and
generic PID arguments hold.  The fixed valuation \(w\) now has
residue transcendence degree infinity over \(L\), so
\(\tau_A(0,h)=\infty\).  In the last flag bound simply use
\(\min(n,\tau)\le n\); all other cases are unchanged.  This
proves the theorem for infinite capacity as well.
\end{proof}

\section{Association extensions and finite maxima}\label{sec:gluing}

\begin{definition}
A faithfully flat local inclusion of domains \(A\subset S\) is an
\emph{association extension} if every nonzero \(s\in S\) has the
form \(s=au\), with \(a\in A\) and \(u\in S^\times\).
\end{definition}

\begin{lemma}\label{lem:association}
For an association extension \(A\subset S\), extension and
contraction identify the ideal lattices and the prime spectra.
The property persists under corresponding prime quotients and
localizations. If \(A\) is a UFD, then \(S\) is a UFD, and
\(\dim S[n]=\dim A[n]\) for every \(n\ge0\).
\end{lemma}
\begin{proof}
If \(J\) is an ideal of \(S\) and \(s=au\in J\), then
\(a\in J\cap A\). Hence \(J=(J\cap A)S\). Faithful flatness
gives \(IS\cap A=I\) for every ideal \(I\) of \(A\). It also
supplies a prime above every prime \(p\) of \(A\); that prime
must be \(pS\). Quotients and localizations preserve both
association and faithful flatness. Prime elements of a UFD \(A\)
therefore remain prime in \(S\), and factoring the associated
element \(a\) proves factoriality of \(S\).

For corresponding primes \(p<q\), write
\(B=(A/p)_{q/p}\) and \(T=(S/pS)_{qS/pS}\). Let \(W\)
be a valuation of \(\operatorname{Frac}T\) dominating \(T\),
and let \(V\) be its restriction to \(\operatorname{Frac}B\).
Every nonzero fraction is \(z=(a/b)u\), with \(a,b\in B\)
and \(u\in T^\times\). If \(z\) has value zero, so does
\(a/b\), and its residue lies in the compositum
\(\kappa(V)\kappa(T)\) inside \(\kappa(W)\). Conversely both
fields are subfields of \(\kappa(W)\). Thus
\[
 \kappa(W)=\kappa(V)\kappa(T),\qquad
 \operatorname{trdeg}_{\kappa(T)}\kappa(W)
 \le\operatorname{trdeg}_{\kappa(B)}\kappa(V).
\]
Consequently every capacity of \(S\) is at most the corresponding
capacity of \(A\). If \(\dim A<\infty\), \cref{thm:capacity}
gives \(\dim S[n]\le\dim A[n]\). The reverse inequality
follows from faithful flatness of \(A[n]\to S[n]\): lift a
prime above the top of a finite chain and apply going-down.
If \(\dim A=\infty\), this lower bound already makes both
dimensions infinite.
\end{proof}

\begin{lemma}\label{lem:semilocal-ufd}
A semilocal domain whose maximal localizations are UFDs is a UFD.
\end{lemma}
\begin{proof}
A domain is the intersection of its maximal localizations in its
fraction field. Each UFD is the finite-character intersection of
its prime-element DVRs; the union of finitely many such families
still has finite character. Thus the given domain is Krull.
A height-one prime is principal at each maximal localization.
Choose an element of the prime generating its localization at each
maximal ideal, including a unit generator when that localization
is the unit ideal. These finitely many elements generate the prime,
since equality of ideals is detected at all maximal ideals.
It is therefore an invertible ideal. An invertible ideal of a
semilocal ring is principal: use the Chinese remainder theorem
to choose a generator modulo every maximal ideal and apply
Nakayama's lemma locally. Hence every height-one prime is
principal, and the Krull domain is factorial.
\end{proof}

\begin{theorem}\label{thm:separation}
Let \(R\) be a countable semilocal UFD containing \(\mathbb Q\),
with maximal ideals \(\mathfrak m_1,\ldots,\mathfrak m_s\).
There is a countable semilocal UFD \(S\supset R\), with precisely
\(s\) maximal ideals \(M_i\) contracting to \(\mathfrak m_i\),
such that every \(R_{\mathfrak m_i}\subset S_{M_i}\) is an
association extension and every nonzero prime of \(S\) is
contained in exactly one maximal ideal.
\end{theorem}
\begin{proof}
If \(s=1\), take \(S=R\). Assume \(s\ge2\), and choose
\(0\ne\pi_i\in\mathfrak m_i\setminus\bigcup_{j\ne i}
\mathfrak m_j\) by the Chinese remainder theorem.

\emph{Finite stages.}
For each prime element \(p\) of the original UFD \(R\) lying
in at least two maximal ideals, let
\(I(p)=\{i:p\in\mathfrak m_i\}\), and introduce a block
\begin{equation}\label{eq:product-block}
 \prod_{i\in I(p)}t_{p,i}=p.
\end{equation}
At the \(i\)-th center with \(i\in I(p)\), give \(t_{p,i}\)
residue zero and every other factor residue one. Outside
\(I(p)\), give one fixed distinguished factor residue \(\bar p\)
and the others residue one. Use finitely many independent
blocks and localize outside the union of these \(s\) centers.
The resulting ring \(T\) is a domain: over a domain and a
nonzero coefficient \(p\), the product relation has the usual
normal forms with no monomial divisible by the product of all
variables, which embed after eliminating one variable in the
rational function field. Apply this successively to the blocks.

At each center, all factors but its distinguished factor are
units, and that factor is eliminated by \eqref{eq:product-block}.
The local ring is precisely
\begin{equation}\label{eq:flat-chart}
 T_i=R_{\mathfrak m_i}[\text{free unit coordinates}]
 _{(\mathfrak m_i,\text{coordinates}-1)}.
\end{equation}
It is faithfully flat over the fixed original \(R_{\mathfrak m_i}\)
and is a UFD. Thus \(T\) is a semilocal UFD by
\cref{lem:semilocal-ufd}. Adjoining a further independent block
gives an injective map preserving all centers.

\emph{A simultaneous coordinate refinement.}
Hold all but one block fixed and write the moving block as
\(t_1\cdots t_k=p\), \(k\ge2\). At center \(i\), denote
its distinguished index by \(d(i)\). Prescribe unit values
\(\lambda_{ij}\in1+\mathfrak m_iR_{\mathfrak m_i}\) for
\(j\ne d(i)\), and define local units
\begin{equation}\label{eq:desired-ratios}
 h_{ij}=\lambda_{ij}/t_j\quad(j\ne d(i)),\qquad
 h_{i,d(i)}=
 \frac{\prod_{j\ne d(i)}t_j}{\prod_{j\ne d(i)}\lambda_{ij}}.
\end{equation}
Their product is one, and each has residue one.
Choose nonzero \(\rho_i\in T\) supported only at the \(i\)-th
center, monomial in each moving block up to a coefficient in
the other blocks. Then \(T/(\rho_i)\) is local. In particular,
the image of each \(h_{ij}\in T_i\) is defined in that quotient.
The ideals \((\rho_i)\) are pairwise comaximal, so for \(j<k\)
the Chinese remainder theorem supplies \(H_j^0\in T\) with
\[
 H_j^0-h_{ij}\in\rho_iT_i\qquad(1\le i\le s).
\]
Every \(H_j^0\) is a unit of \(T\).

Set \(L=\prod_i\rho_i\), and choose \(M\ge0\) and rational
\(\theta_1,\ldots,\theta_{k-1}\). In a fresh copy of the same
abstract stage, define the old coordinates by
\begin{align}
 H_j&=H_j^0+\theta_jLt_j^M,\qquad
 t_j^{\mathrm{old}}=t_jH_j &&(j<k),\label{eq:refine-product}\\
 t_k^{\mathrm{old}}&=t_k/(H_1\cdots H_{k-1}).\notag
\end{align}
The added terms lie in every \(\rho_iT_i\). Hence all \(H_j\)
remain units; the product relation, centers, and residue maps
are preserved. Since \(\prod_jh_{ij}=1\), every free coordinate
satisfies
\begin{equation}\label{eq:coordinate-precision}
 t_j^{\mathrm{old}}-\lambda_{ij}\in\rho_iT_i
 \qquad(j\ne d(i)).
\end{equation}

The parameters can be chosen to make this map injective.
Let \(F\) be the fraction field of the original coefficients
and fixed blocks. The first \(k-1\) moving coordinates are
independent over \(F\), and
\(L=C\prod_{j<k}t_j^{\beta_j}\), with \(C\in F^\times\)
and integers \(\beta_j\). The coefficient of
\(\prod_{j<k}\theta_j\) in the Jacobian of the first \(k-1\)
functions in \eqref{eq:refine-product} is the Jacobian of
\(G_j=Lt_j^{M+1}\). Up to a nonzero scalar and Laurent monomial
it equals
\begin{equation}\label{eq:jacobian}
 \det((M+1)I+\mathbf1\beta^{\mathsf T})
 =(M+1)^{k-2}\left(M+1+\sum_{j<k}\beta_j\right).
\end{equation}
Choose \(M\) large enough to make this nonzero. A nonzero
polynomial over a field cannot vanish on all of
\(\mathbb Q^{k-1}\); choose \(\theta\) giving a nonzero
Jacobian. In characteristic zero the resulting functions are
algebraically independent. Thus the field map, and therefore
the centered map of semilocal stages, is injective.

\emph{Processing one element.}
Let \(0\ne f\in T\). In each chart \eqref{eq:flat-chart},
express it as a quotient of Laurent polynomials, with numerator
\(h_i\ne0\) and unit denominator. Choose the finitely many
free coordinates independently in \(1+\pi_i\mathbb Q\) so
that \(a_i=h_i(\lambda_i)\ne0\). The denominator remains a
unit because the values have the prescribed residues.
Write \(a_i=r_i/u_i\), with \(r_i,u_i\in R\) and
\(u_i\notin\mathfrak m_i\). Factor \(r_i\) in the original
UFD \(R\), and introduce missing blocks for its shared prime
factors. Replace each factor \(p\in\mathfrak m_i\) by
\(t_{p,i}\) when \(p\) is shared, retaining \(p\) when it
is supported only there. Omit factors outside \(\mathfrak m_i\).
The resulting product \(\mu_i\) is associated to \(a_i\)
in \(T_i\) and is a unit at every other center. Put
\[
 \rho_i=\pi_i\mu_i^2\in a_i^2\mathfrak m_iT_i.
\]
These are monomials of the required kind.

Refine each of the finitely many old blocks in turn using
\eqref{eq:coordinate-precision}, holding the others fixed.
The values \(a_i\) lie in the fixed original base. Therefore
every later centered map preserves the earlier congruences
modulo \(a_i^2\mathfrak m_i\), even when it changes a factor
appearing in a newly chosen \(\rho_i\). At the end,
\[
 h_i=a_i+a_i^2z_i=a_i(1+a_i z_i),\qquad
 z_i\in\mathfrak m_iT_i.
\]
The last factor and the original denominator are units. Thus
\(f\) is associated to the original \(a_i\) at every center.
This is an exact finite identity preserved by all later maps.

\emph{The limit.}
Enumerate every shared original prime and every element of every
created stage. Interleave adjoining the missing blocks and the
finite processing just described so that every task is reached.
The injective direct limit \(S\) is countable. Its preserved
residue maps give precisely \(s\) maximal ideals: any element
outside their union is a unit already at a finite stage.
Each \(S_{M_i}\) is a filtered union of the fixed-base flat
charts \eqref{eq:flat-chart}, so it is faithfully flat over
\(R_{\mathfrak m_i}\). Every element is eventually associated
to an original element. Hence its localization is a UFD by
\cref{lem:association}, and \(S\) is a UFD by
\cref{lem:semilocal-ufd}.

Finally let \(0\ne P\in\operatorname{Spec}S\), and choose
\(M_i\supseteq P\). Local association of a nonzero element
of \(P\), followed by contraction from \(S_{M_i}\), puts a
nonzero original numerator in \(P\). Thus \(P\) contains an
original prime factor \(p\). If \(p\) originally had one
support, it remains a unit at every other center. Otherwise
\eqref{eq:product-block} forces one of its singly supported
factors into \(P\). The images of these factors retain their
supports under every refinement, since they are multiplied by
units. Therefore \(P\) can lie under only one maximal ideal.
\end{proof}

\begin{lemma}\label{lem:birational-association}
Let \(A\subset S\) be an association extension and let
\(A\subset B\subset\operatorname{Frac}A\), with \(B\)
local. Then \(T=S\otimes_A B\) embeds in
\(\operatorname{Frac}S\), is local, and \(B\subset T\) is
an association extension. In particular, a UFD \(B\) gives
a UFD \(T\) with the same polynomial dimensions.
\end{lemma}
\begin{proof}
Flatness embeds \(T\) in the localization of \(S\) at
\(A\setminus\{0\}\), which association identifies with
\(\operatorname{Frac}S\). The map \(B\to T\) is faithfully
flat. Moreover \(T\cap\operatorname{Frac}A=B\): if
\(c/d\in T\), with \(c,d\in B\), \(d\ne0\), then
\(c\in dT\cap B=dB\). Every element of \(T\) is \(s/r\),
with \(s\in S\) and \(0\ne r\in A\), by combining its
finitely many denominators. Write \(s=au\) by association.
Then \(a/r\in T\cap\operatorname{Frac}A=B\), so \(s/r\)
is associated to an element of \(B\). The ideal-lattice argument
in \cref{lem:association} now shows that \(T\) has the unique
maximal ideal \(\mathfrak m_BT\) and proves the remaining claims.
\end{proof}

\begin{lemma}\label{lem:rational-padding}
Let \(B\) be a countable local domain containing \(\mathbb Q\)
and having a nonzero element \(\pi\in\mathfrak m_B\).
For a finite or countable list of independent variables \(Z\),
there is an association extension \(B\subset B'\) containing
\(B[Z]_{(\mathfrak m_B,Z)}\) locally and having fraction field
\(\operatorname{Frac}(B)(Z)\).
\end{lemma}
\begin{proof}
Use finite stages \(B[t_1,\ldots,t_k]_{(\mathfrak m_B,t)}\).
For a nonzero numerator \(f(t)\), choose rational \(c_j\)
with \(a=f(\pi c_1,\ldots,\pi c_k)\ne0\), and replace
\[
 t_j=\pi c_j+a^2\pi T_j\qquad(1\le j\le k).
\]
These are injective local maps and invertible affine changes
over \(\operatorname{Frac}B\), so they preserve the specified
rational function field. In the new stage
\(f=a(1+a\pi h)\), with the second factor a unit.
Process all numerators fairly, introducing fresh independent
variables when needed. Each stage is flat over the fixed \(B\);
the local union is faithfully flat and has association for every
element. Previously obtained unit factors remain units under
the local maps. This proves the lemma.
\end{proof}

\begin{proposition}\label{prop:exact-gluing}
Let \(S\) be a semilocal UFD with maximal ideals \(M_i\),
such that every nonzero prime has one maximal support. Suppose
\(S_{M_i}\subset T_i\subset\operatorname{Frac}S\) are local
UFDs whose maximal ideals contract to those of \(S_{M_i}\).
Then \(T=\bigcap_iT_i\) is a semilocal UFD and
\[
 T_{N_i}=T_i,\qquad N_i=T\cap\mathfrak m_{T_i},\qquad
 \dim T[n]=\max_i\dim T_i[n].
\]
\end{proposition}
\begin{proof}
The nonunits of \(T\) are the union of the \(N_i\).
Their contractions to \(S\) are the distinct maximal ideals
\(M_i\), so the \(N_i\) are incomparable. Finite prime
avoidance makes them precisely the maximal ideals of \(T\).
For \(x\in T_i\), write \(x=a/b\) with \(a,b\in S\),
\(b\ne0\), and factor \(b=b_i c\), putting exactly its
prime factors supported at \(M_i\) into \(b_i\). Then
\(c\notin M_i\), and \(b_i\) is a unit at every other
\(M_j\). Thus \(cx=a/b_i\in S_{M_j}\subset T_j\) for
\(j\ne i\), and also \(cx\in T_i\). Therefore \(cx\in T\)
and \(c\notin N_i\), proving \(T_{N_i}=T_i\).
Factoriality follows from \cref{lem:semilocal-ufd}.
Every finite polynomial prime chain survives in a maximal
localization containing the contraction of its top prime;
conversely, a chain in a localization contracts without
collapsing. This gives the displayed dimension formula.
\end{proof}

\section{Assembly and consequences}\label{sec:assembly}

\begin{theorem}\label{thm:mixed}
Every finite maximum of the ordinary profiles
\eqref{eq:ordinary-profile} and delayed profiles
\eqref{eq:delayed-profile} is realized by a countable semilocal
UFD containing \(\mathbb Q\).
\end{theorem}
\begin{proof}
Replace a delayed profile with zero caps by the two ordinary
profiles in \cref{cor:krull-envelope}. For an ordinary target,
prescribe that same profile as its lower profile. For a delayed
target with positive caps \(t_1,\ldots,t_r\), prescribe
\[
 c(n)=n+r+\sum_{i=2}^r\min(n,t_i).
\]
\Cref{thm:ordinary} provides one countable semilocal UFD \(C\)
whose maximal localizations \(C_j\) have precisely these lower
profiles and the arithmetic and valuation hypotheses of
\cref{thm:relative}. Their fraction field is the common
\(F=\operatorname{Frac}C\).

At a delayed branch apply \cref{thm:relative} to \(C_j\) with
capacity \(t_1\). The resulting local UFD \(A_j\) satisfies
\[
 \dim A_j[n]=c(n)+\max\{1,\min(n,t_1)\}.
\]
At \(n=0\) this is \(r+1\). For \(n\ge1\), positivity
of all caps makes it
\(n+r+\sum_{i=1}^r\min(n,t_i)\), which is exactly the
delayed profile. At an ordinary branch put \(A_j=C_j\).

Choose one common finite or countably infinite list \(Z\) of
indeterminates large enough for all upper rows. At each branch
name its active variables using an initial portion of this list.
The relative construction retains its initial polynomial local
ring and adjoins explicit fractions in its fixed rational field;
its fraction field is therefore precisely \(F\) with those
active variables adjoined. Apply \cref{lem:rational-padding}
over the completed \(A_j\) to the unused variables. The new
local UFD \(B_j\) has the same profile as \(A_j\), contains
\[
 R_j=C_j[Z]_{(\mathfrak m_{C_j},Z)}
\]
locally, and has fraction field \(F(Z)\). In particular every
\(B_j\) is birational over its \(R_j\). Countably many
variables cause no change in this argument: every fraction
and every identity involves only finitely many variables.

Let \(R\) be the semilocalization of \(C[Z]\) outside the
union of the ideals \((\mathfrak m_j,Z)\). It is a countable
semilocal UFD with maximal localizations \(R_j\). Apply
\cref{thm:separation}, obtaining \(S\supset R\). This theorem
does not require \(R\) to be finite-dimensional, so infinite
\(Z\) is allowed. By \cref{lem:birational-association},
\[
 T_j=S_{M_j}\otimes_{R_j}B_j\subset\operatorname{Frac}S
\]
is a local UFD with the profile of \(B_j\). Its maximal ideal
contracts to that of \(S_{M_j}\): for \(s=au\in S_{M_j}\),
association reduces membership to \(a\in\mathfrak m_{B_j}
\cap R_j=\mathfrak m_{R_j}\).
\Cref{prop:exact-gluing} now makes \(T=\bigcap_jT_j\) a
semilocal UFD with
\(\dim T[n]=\max_j\dim T_j[n]\). It is countable as a
subring of the countable field \(\operatorname{Frac}S\), and
contains \(\mathbb Q\). This is the required maximum.
\end{proof}

\begin{proof}[Proof of \cref{thm:main}]
Necessity follows from \cref{prop:necessary}, since every UFD
is Krull. Conversely, if \(d=0\), the first inequality forces
\(a_n=n\), which is realized by a field. For \(d\ge1\),
\cref{cor:krull-envelope} expresses the prescribed sequence as a finite
maximum of ordinary and delayed profiles. \Cref{thm:mixed}
constructs a UFD realizing that maximum.
\end{proof}

\begin{proof}[Proof of \cref{cor:bouvier}]
Apply \cref{thm:ordinary} with one branch, \(d_1=1\), and
\(t_{1,1}=1\). The resulting ring \(A\) is a local UFD with
\[
 \dim A[n]=n+2+\min(n,1)\qquad(n\ge0).
\]
In particular, \(\dim A=2\) and \(\dim A[X]=4\).
Every UFD is Krull, and the strict inequality
\(\dim A[X]>\dim A+1\) shows that \(A\) is not Jaffard.
This proves Bouvier's conjecture.
\end{proof}

\begin{corollary}\label{cor:two-dimensional-minimal}
For a sequence with \(a_0=2\) and \(a_1=3\), the following
are equivalent: it is realized by a Krull domain; it is realized
by a countable local UFD containing \(\mathbb Q\); and
\begin{equation}\label{eq:two-dimensional-delayed}
 a_n=n+\max\{2,1+\min(n,t)\}
 \quad(n\ge0)
\end{equation}
for some \(t\in\mathbb N_0\cup\{\infty\}\).
\end{corollary}
\begin{proof}
Let \(A\) be a two-dimensional Krull domain with
\(\dim A[1]=3\). For a height-one prime \(p\) properly
contained in a maximal ideal \(m\), a positive capacity
\(\tau_A(p,m)\) would give, by \cref{thm:capacity}, a
length-four chain in \(A[1]\). Thus every such capacity
vanishes. Also \(\tau_A(0,p)=0\), since \(A_p\) is a DVR.
The remaining nonsaturated flags are \(0<m\), with \(m\)
of height two. If \(t\) is the supremum of their capacities,
\cref{thm:capacity} gives \eqref{eq:two-dimensional-delayed}.
Height-one maximal ideals add no larger value, so this proves
necessity without a locality assumption on \(A\).

For \(t\ge1\), apply \cref{thm:relative} over
\(\mathbb Q[\pi]_{(\pi)}\). For \(t=0\), use
\(\mathbb Q[x,y]_{(x,y)}\). The cases \(t=0,1\) give
the same profile; every other parameter is determined by the
length of the run of increments equal to two.
\end{proof}

\begin{corollary}[The quantitative Bouvier--Kabbaj conjecture]\label{cor:bouvier-defect}
For every integer \(r\ge1\) there is a countable local UFD
\(A\supset\mathbb Q\) such that
\[
 \dim A=2,\qquad \dim A[X]=3,\qquad \dim_{\mathrm v}A=2+r.
\]
There is also a countable two-dimensional local UFD with
\(\dim A[X]=3\) and infinite valuative dimension.
\end{corollary}
\begin{proof}
Use \eqref{eq:two-dimensional-delayed} with \(t=r+1\),
respectively \(t=\infty\), and the identity
\(\dim_{\mathrm v}A=\sup_n(\dim A[n]-n)\).
For finite \(t\), the displayed excess has supremum
\(t+1=r+2\); for infinite \(t\), it is unbounded.
\end{proof}

\begin{corollary}\label{cor:two-dimensional-full}
The two-dimensional Krull/UFD profiles are exactly
\[
 a_n=n+\max\{2+\min(n,T),1+\min(n,U)\},
 \qquad T,U\in\mathbb N_0\cup\{\infty\}.
\]
Every such profile has a countable semilocal UFD realization.
\end{corollary}
\begin{proof}
In \cref{cor:krull-envelope}, every ordinary support below
initial dimension two is dominated by one of the profiles
\(n+2+\min(n,T)\); every delayed support has depth one and
is dominated by \(n+\max\{2,1+\min(n,U)\}\). Each family
is ordered by its parameter. Taking the largest parameters
therefore gives the displayed form. Conversely each displayed
maximum is realized by \cref{thm:mixed}.
\end{proof}

For instance, \((T,U)=(1,3)\) gives
\((2,4,5,7,8,9,\ldots)\). The unrestricted maximal envelope
is \(d+(d+1)n\), whereas the Krull/UFD maximal envelope is
\(d(n+1)\) for \(d\ge1\). Both are attained in their respective classes.
Finite products suffice for the unrestricted envelope construction;
the association and support-separation arguments are what allow
the same numerical operation to be performed within domains
with unique factorization.

\renewcommand{\sectionname}{}
\section*{Acknowledgments}
We used GPT-5.6 Sol and an agentic harness built around GPT-5.6 Sol and Claude Fable 5 to assist with literature searches, hypothesis testing, the exploration and elimination of potential approaches, wording refinement, and manuscript proofreading. We thank Hieu M. Vu, Tho Tran Huu, Khoi M. N. Nguyen, Dung V. Nguyen, and Quang X. Nguyen for their assistance with hardware-related matters and for providing technical support in the use of the AI tools and agentic system.

\bibliographystyle{plain}
\begingroup
\footnotesize
\raggedright
\bibliography{references}
\endgroup

\end{document}